\documentclass{amsart}
\usepackage{verbatim}
\usepackage{amssymb, amsmath}
\usepackage{graphicx}
\usepackage{caption}
\usepackage{subcaption}
\usepackage[all]{xy}
\usepackage{graphicx}
\usepackage{mathtools}
\usepackage{fancyhdr}
\usepackage{hyperref}
\usepackage{color}

\begin{document}
\newtheorem{thm}{Theorem}[section]
\newtheorem*{thm*}{Theorem}
\newtheorem{lem}[thm]{Lemma}
\newtheorem{prop}[thm]{Proposition}
\newtheorem{cor}[thm]{Corollary}
\newtheorem*{conj}{Conjecture}
\newtheorem{proj}[thm]{Project}
\newtheorem{question}[thm]{Question}
\newtheorem{rem}{Remark}[section]

\theoremstyle{definition}
\newtheorem{defn}{Definition}
\newtheorem*{remark}{Remark}
\newtheorem{exercise}{Exercise}
\newtheorem*{exercise*}{Exercise}
\numberwithin{equation}{section}

\newcommand{\Li}{\operatorname{Li}}

\title[Eigenvalue Statistics on Random Covers] {Smooth Linear Eigenvalue Statistics on Random Covers of Compact Hyperbolic Surfaces -- A Central Limit Theorem and Almost Sure RMT Statistics}
\author{Yotam Maoz}
\address{School of Mathematical Sciences, Tel Aviv University, Tel Aviv 69978, Israel}
\email{yotammaoz@mail.tau.ac.il}

\date{\today}
\begin{abstract}
	We study smooth linear spectral statistics of twisted Laplacians on random $n$-covers of a fixed compact hyperbolic surface $X$.
	We consider two aspects of such statistics. The first is the fluctuations of such statistics in a small energy window around a fixed energy level when averaged over the space of all degree $n$ covers of $X$.
	The second is the centered energy variance of a typical surface, a quantity similar to the normal energy variance.

	In the first case, we show a central limit theorem. Specifically, we show that the distribution of such fluctuations tends to a Gaussian with variance given by the
	corresponding quantity for the Gaussian Orthogonal/Unitary Ensemble (GOE/GUE). In the second case, we show that the centered energy variance of a typical random $n$-cover is that of the GOE/GUE.
	In both cases, we consider a double limit where first we let $n$ - the covering degree, go to $\infty$ then let $L\to \infty$ where $1/L$ is the window length.

	A fundamental component of our proofs are the results we prove in \cite{surfaces}
	which concern the random cover model for random surfaces.
\end{abstract}

\maketitle

\pagebreak

\tableofcontents

\pagebreak

\section{Introduction}
Let $X$ be a compact hyperbolic surface of genus $g\geq2$, and fundamental
group $\Gamma\stackrel{\text{def}}{=}\left\langle a_{1},b_{1},...,a_{g},b_{g}|[a_{1},b_{1}]...[a_{g},b_{b}]\right\rangle $ (which we sometimes view as a subgroup of $PSL_{2}(\mathbb{R}))$.
In addition, let $\chi:\Gamma\to \mathbb{C}$  be a complex one-dimensional unitary
representation.
We denote by $\Delta_{\chi}$ the Laplacian on $X$ twisted by $\chi$. That is, $\Delta_{\chi}$ is the usual
hyperbolic Laplacian acting on smooth functions $f:\mathbb{H}\to \mathbb{C}$
satisfying the equivariance property $f(\gamma z)=\chi(\gamma)f(z)$
for all $\gamma\in\Gamma$. The space of all such functions is equipped
with an $L^{2}$ norm defined as:
\[
	\Vert f\Vert^{2}=\int_{\mathcal{F}}\Vert f(z)\Vert^{2}dVol(z),
\]
where $\mathcal{F}$ is some compact fundamental domain for $\Gamma$ and $Vol$ is hyperbolic volume.
The twisted Laplacian $\Delta_{\chi}$ has a unique self-adjoint
extension with a discrete spectrum, which is the central object of this
paper.

Let $o\in X$ be a point. Recall that elements of $\text{Hom}(\Gamma,S_{n})$, consisting of
all group morphisms $\Gamma\to S_{n}$, are in a bijection with $n$-sheeted
covers of $X$ with a labeled fiber $\{1,...,n\}$ of $o$. See \cite{MP1} and Section \ref{bkgrd} below. From now on we denote by $X_{\phi}$ the cover of $X$ corresponding to a $\phi \in \text{Hom}(\Gamma,S_{n})$.
Thus, by considering the uniform probability measure
on the finite\footnote{One can show that for a finite group $G$ we have the connection:
\[
	\#\text{Hom}(\Gamma,G) = |G|^{2g-1}\zeta^{G}(2g-2),
\]
where $\zeta^{G}(s) = \sum_{\rho\in \text{Irrep}(G)}\dim(\rho)^{-s}$
is the Witten zeta function of $G$ and $\text{Irrep}(G)$ denotes the set of complex irreducible representations of $G$. See \cite[Proposition 3.2]{LS}.
} set $\text{Hom}(\Gamma,S_{n})$ one obtains a notion
of a random $n$-sheeted covering of $X$. We denote by $\mathbb{E}_{n}[\cdot]$
the expected value operator on this space, that is:
\[
	\mathbb{E}_{n}[T]=\frac{1}{\#\text{Hom}(\Gamma,S_{n})}\sum_{\phi\in\text{Hom}(\Gamma,S_{n})}T(\phi),
\]
where $T$ is some random variable on this space.

As $n$-sheeted covers of $X$ are also compact hyperbolic surfaces,
each $n$-sheeted cover $X_{\phi}$ of $X$ comes with its own twisted
Laplacian $\Delta_{\phi,\chi}$ and thus its own spectrum. We denote
the spectrum of $\Delta_{\phi,\chi}$ by $\{\lambda_{\phi,\chi,j}\}_{j\geq 0}$
(counted with multiplicity) and fix, for all $j\geq 0$, an element:
\[
	r_{\phi,\chi,j}\in\mathbb{R}\cup i\mathbb{R},
\]
such that $\lambda_{\phi,\chi,j}=1/4+(r_{\phi,\chi,j})^{2}$.
Also, note that as $n\to \infty$ asymptotically almost all $X_{\phi}$ are connected \cite[Theorem 1.12]{LS}.
\par

Let $\psi$ be an even function whose Fourier transform $\hat{\psi}$
is smooth and compactly supported on $[-1,1]$. In particular, $\psi$ extends to an entire function.
Our notion of the Fourier transform is:
\par
\[
	\hat{\psi}(s)=\frac{1}{2\pi}\int_{\mathbb{R}}\psi(x)e^{-isx}dx.
\]

Let $\alpha,L>0$ and set:
\[
	h(r)=\psi(L(r-\alpha))+\psi(L(r+\alpha)).
\]
We consider the smooth
counting function:
\[
	N_{\phi}(L)=\sum_{j\geq0}h(r_{\phi,\chi,j}),
\]
defined for any $n$-sheeted cover $X_{\phi}$ of $X$ corresponding
to an element
\newline $\phi\in\text{Hom}(\Gamma,S_{n})$.
Note that for large $\alpha, L\gg1$ the quantity $N_{\phi}(L)$ is a proxy for the number of eigenvalues
in a window of size $\approx 1/L$ around $\alpha^{2}$. One can, therefore, consider $N_{\phi}(L)$
as a random variable over the space of $n$-sheeted covers of $X$. To avoid confusion,
we let $N_{n}(L)$ denote the value of $N_{\phi}(L)$ for a random $\phi\in\text{Hom}(\Gamma,S_{n})$.
Thus, $N_{n}(L)$ is a random variable on the space $\text{Hom}(\Gamma,S_{n})$, while $N_{\phi}(L)$ always stands for the value of this statistic for a given $\phi\in\text{Hom}(\Gamma,S_{n})$.

\subsection{Statistics for a Fixed Height and a Random Surface}
The first natural question to ask of $N_{n}(L)$ is the following:
\begin{question}\label{q1}
	If $\alpha$, the window height, is fixed, what can be said of the distribution of $N_{n}(L)$?
\end{question}

A natural place to start trying to answer Question \ref{q1} is to consider the expected value $\mathbb{E}_{n}\left[N_{n}(L)\right]$. Using the twisted trace formula, Theorem \ref{trace}, in the next section, one finds that as $n\to\infty$ we have:
\[
	\mathbb{E}_{n}\left[N_{n}(L)\right] \sim C_{\alpha}\frac{(g-1)n}{L}\int_{\mathbb{R}}\psi(r)dr,
\]
where $C_{\alpha} = 2\alpha \tanh(\pi\alpha)$.
In particular, if $\int_{\mathbb{R}}\psi(r)dr \neq 0$, then the expected value tends to $\pm \infty$ as $n\to\infty$.
This suggests one should consider fluctuations about the mean $\mathbb{E}_{n}\left[N_{n}(L)\right]$.

Let $\textnormal{Var}_{n}$ denote the variance operator on the space $\text{Hom}(\Gamma,S_{n})$, formally:
\[
	\textnormal{Var}_{n}[T] = \mathbb{E}_{n}\left[\left(T - \mathbb{E}_{n}[T]\right)^{2}\right],
\]
where $T$ is some random variable on the space.
Recently, Naud \cite{Naud} studied the variance of $N_{n}(L)$, he showed:

\begin{thm}[Naud 2022] \label{Naud main}Let $X$ and $\chi$ as before,
	in addition fix $\alpha\in\mathbb{R}$, then:
	\[
		\underset{L\to\infty}{\lim}\underset{n\to\infty}{\lim}\textnormal{Var}_{n}\left[N_{n}(L)\right]=\begin{cases}
			\Sigma_{\text{GOE}}^{2}(\psi) & \chi^{2}=1,    \\
			\Sigma_{\text{GUE}}^{2}(\psi) & \chi^{2}\neq1.
		\end{cases}
	\]
	where $\Sigma_{\text{GOE}}^{2}(\psi)$ is the ``smoothed'' number
	variance of random matrices for the $\text{GOE}$ model in the large
	dimension limit and is given by:
	\[
		\Sigma_{\text{GOE}}^{2}(\psi)=2\int_{\mathbb{R}}\left|x\right|\left[\hat{\psi}(x)\right]^{2}dx,
	\]
	and $\Sigma_{\text{GUE}}^{2}(\psi)=\frac{1}{2}\Sigma_{\text{GOE}}^{2}(\psi)$.
\end{thm}

Our first main theorem concerns the higher central moments of $N_{n}(L)$. Define:
\[
	\mathbb{V}_{n}^{(k)}(L)=\mathbb{E}_{n}\left[\left(N_{n}(L)-\mathbb{E}_{n}\left[N_{n}(L)\right]\right)^{k}\right].
\]
We wish to study $\mathbb{V}_{n}^{(k)}(L)$ in the same setting as Naud in \cite{Naud}, that is, we wish to find:
\[
	\underset{L\to\infty}{\lim}\underset{n\to\infty}{\lim}\mathbb{V}_{n}^{(k)}(L).
\]
As our first main theorem, we show:

\begin{thm} \label{main thm 1}  Let $X$ and $\chi$ as before, and in addition fix
	$\alpha\in\mathbb{R}$, then for all $k\geq 2$:
	\[
		\underset{L\to\infty}{\lim}\underset{n\to\infty}{\lim}\mathbb{V}_{n}^{(k)}(L)=
		\begin{cases}
			(k-1)!!\sigma_{\chi,\psi}^{k} & k\text{ even} \\
			0                             & k\text{ odd}
		\end{cases}
	\]
	where:
	\[
		\sigma_{\chi,\psi}^{2}=
		\begin{cases}
			\Sigma_{\text{GOE}}^{2}(\psi) & \chi^{2}=1,    \\
			\Sigma_{\text{GUE}}^{2}(\psi) & \chi^{2}\neq1.
		\end{cases}
	\]
\end{thm}

As per Theorem \ref{main thm 1}, central moments of $N_{n}(L)$ are those of a centered Gaussian
with standard deviation $\sigma_{\chi,\psi}$. As a simple corollary, using the "method of moments" (see \cite[Section 30]{BP}), we have a central limit theorem:
\begin{cor} \label{CLT}
	Let $X,\chi,\alpha$ as before, then for all bounded continuous $g$ we have:
	\[
		\underset{L\to\infty}{\lim}\underset{n\to\infty}{\lim}\mathbb{E}_{n}\left[g\left(\frac{N_{n}(L)-\mathbb{E}_{n}\left[N_{n}(L)\right]}{\sigma_{\chi,\psi}} \right)\right]=\frac{1}{\sqrt{2\pi}}\int_{\mathbb{R}}g(x)e^{-x^{2}/2}dx.
	\]
\end{cor}

\subsection{The Centered Energy Variance of a Random Surface}
Given an $n$-cover $X_{\phi}$ of $X$, one can vary the window height $\alpha$ and consider the fluctuations of $N_{\phi}(L)$ as $\alpha$ varies and $\phi$ remains constant.

Formally, let $w$ be a non-negative even weight function satisfying $\int_{\mathbb{R}}w=1$,
with smooth and compactly supported Fourier transform $\hat{w}$.
For $T>0$ define the following expected value operator:

\[
	\mathbb{E}_{T}[F]=\frac{1}{T}\int_{\mathbb{R}}F(\alpha)w\left(\alpha/T\right)d\alpha,
\]
and the corresponding variance:
\[
	\mathbb{V}_{T}[F]=\mathbb{E}_{T}\left[\left(F-\mathbb{E}_{T}[F]\right)^{2}\right].
\]

For $\phi\in\text{Hom}(\Gamma,S_{n})$, we set the centered energy variance of $X_{\phi}$ to be:
\[
	\mathbb{V}_{T,L}(X_{\phi}) = \mathbb{V}_{T}\left[N_{\phi}(L) - \mathbb{E}_{n}\left[N_{n}(L)\right]\right],
\]
where the variance is taken with respect to $\alpha$.
We view $\mathbb{V}_{T,L}(X_{\phi})$ as a random variable on the space $\text{Hom}(\Gamma,S_{n})$.
As before, to avoid confusion we let $\mathbb{V}_{T,L,n}$ denote the value of $\mathbb{V}_{T,L}(X_{\phi})$ for a random $\phi \in \text{Hom}(\Gamma,S_{n})$, in particular:
\[
	\mathbb{V}_{T,L,n} = \mathbb{V}_{T}\left[N_{n}(L) - \mathbb{E}_{n}\left[N_{n}(L)\right]\right].
\]
As such, $\mathbb{V}_{T,L,n}$ is a random variable on the space $\text{Hom}(\Gamma,S_{n})$, while $\mathbb{V}_{T,L}(X_{\phi})$ will always denote the centered energy variance of the surface $X_{\phi}$ for a given $\phi \in \text{Hom}(\Gamma,S_{n})$.

A conjecture of Berry \cite{BR2, BR1} states that for a $\textit{generic}$ fixed surface $X$, the energy variance of $X$ (a quantity similar\footnote{
	Using the Twisted trace formula, Theorem \ref{trace} in the next section, the energy variance of $X_{\phi}$ is given by:
	\[
		\mathbb{V}_{T}\left[N^{\text{osc}}_{\phi}(L) \right]=
		\mathbb{V}_{T}\left[N^{\text{osc}}(h;\phi) \right],
	\]
	in addition, in the proof of Theorem \ref{almost sure}  we will see that:
	\[
		\mathbb{V}_{T,L,n} = \mathbb{V}_{T}\left[N^{\text{osc}}_{n}(L) - \mathbb{E}_{n}\left[N^{\text{osc}}_{n}(L)\right]\right],
	\]
	so the two quantities are rather similar.
} to our centered energy variance)  converges to $\sigma_{\chi,\psi}^{2}$ as $T\to \infty$ and $L\to \infty$ but $L=o(T)$.
As this question is quite intractable at the moment for a fixed surface, we consider a random version adapted to our centered energy variance. Denote the uniform probability
measure on $\text{Hom}(\Gamma,S_{n})$ by $\mathbb{P}_{n}$, one can ask:
\begin{question}\label{q2}
	Let $\epsilon > 0$. What is the probability that $\mathbb{V}_{T,L,n}$, the centered energy variance of a random $n$-cover of $X$, is at least $\epsilon$ away from $\sigma_{\chi,\psi}^{2}$?
\end{question}

We give the following answer which is our second main theorem:
\begin{thm} \label{almost sure}
	Let $X$,$\chi$ as before, for $L=o(T)$ and $\epsilon>0$ such that \newline $\frac{1}{\epsilon} = o\left(\sqrt{L}\right)$ we have:
	\[
		\underset{\begin{smallmatrix}L,T\to\infty\\L=o(T)\end{smallmatrix}}{\lim}\underset{n\to\infty}{\limsup} \,\mathbb{P}_{n}\left(\left|\mathbb{V}_{T,L,n}-\sigma_{\chi,\psi}^{2}\right|\geq\epsilon\right)=0.
	\]
\end{thm}
In particular, the random variable $\mathbb{V}_{T,L,n}$ converges in probability to the constant $\sigma_{\chi,\psi}^{2}$ when $L=o(T)$.

\subsection{Related Results}
In this paper, our model of a ``random surface'' is of a random
$n$-cover of some base surface $X$ for large $n$. There is, however,
another natural model of random surfaces in which one could ask similar
questions about the behavior of a similarly defined smooth linear statistic. For a given hyperbolic surface
$X$ of genus $g\geq2$, one can endow the moduli space $\mathcal{M}_{g}$
of $X$ with a natural measure called the Weil-Petersson measure from which one could extract a probability measure on $\mathcal{M}_{g}$.
Recently, Rudnick \cite{Ru} considered an analog of Question \ref{q1} is this model. He showed that first letting $g\to \infty$ then $L\to \infty$, an analogously defined statistic for the non-twisted Laplacian (where $\chi = 1$) has variance $\Sigma_{\text{GOE}}^{2}(\psi)$
when averaged over the moduli space $\mathcal{M}_{g}$.
Even more recently, Rudnick and Wigman \cite{RW} showed that in the same double limit, we have a central limit theorem analogous to Theorem \ref{CLT} in the Weil-Petersson model.

In addition, Rudnick and Wigman \cite{RW2} considered an analog of Question \ref{q2} in the Weil-Petersson model. They showed a result analogous to Theorem \ref{almost sure}. Explicitly, the energy variance of a random surface converges in probability to $\Sigma_{\text{GOE}}^{2}(\psi)$ in the double limit $\underset{\underset{L,T\to \infty}{L=o(\log T)}}{\lim} \underset{g \to \infty}{\limsup}$.

Note that a degree $n$ cover of $X$ has genus $1+n(g-1)$ so that the limit $n\to\infty$ in our model corresponds to the limit $g\to \infty$ in the Weil-Petersson model.

\subsection{Overview of the Paper}
We start by giving a bit of background which will be of use to us in the proofs of both Theorem \ref{main thm 1} and Theorem \ref{almost sure}.
In what follows we start by proving Theorem \ref{main thm 1} given Theorem \ref{B bound}. See section \ref{main thm 1 proof} for Theorem \ref{B bound}.
Subsequently, we prove Theorem \ref{B bound} in Section \ref{B bound proof}.

Next, we provide an outline of the proof for Theorem \ref{almost sure} in Section \ref{almost sure proof}.
Finally, we fill in the details by proving various lemmas used in the proof.

The reader should note that Sections \ref{main thm 1 proof} and \ref{B bound proof} are independent of Section \ref{almost sure proof} and the Sections following it.

\section*{Acknowledgments}
I would like to thank my advisor Prof. Ze\'{e}v Rudnick for his guidance throughout the process and for suggesting both problems. I would also like to thank the anonymous referee for his helpful comments and suggestions.  This work was supported in part by the Israel Science Foundation (grant No. 1881/20) and the ISF-NSFC joint research program (grant No. 3109/23).

\section{Background} \label{bkgrd}
\subsection{Covers of $X$}
As stated earlier, for an arbitrary $o\in X$ the $n$-sheeted covers of $X$ with a labeled fiber of $o$ are in bijection with
$\text{Hom}(\Gamma,S_{n})$. See \newline \cite[pp. 68-70]{HT} for a comprehensive discussion. To see this bijection, let $\mathbb{H}$
be the standard hyperbolic plane with constant curvature -1. In addition
view $\Gamma$ as $\pi_{1}(X,o)$ and as a subgroup of $PSL_{2}(\mathbb{R})$ - the orientation-preserving automorphism
group of $\mathbb{H}$. Also, view $X$ as $\Gamma\backslash\mathbb{H}$ and set $[n] = \{1,...,n\}$.

Given $\phi\in\text{Hom}(\Gamma,S_{n})$ we define an action of
the discrete $\Gamma$ on $\mathbb{H}\times[n]$ by:
\[
	\gamma(z,j)=(\gamma z,\phi(\gamma)j).
\]
Quotiening $\mathbb{H}\times[n]$ by this action, one finds a (possibly
not connected) $n$-sheeted cover $X_{\phi}=\Gamma\backslash\left(\mathbb{H}\times[n]\right)$
of $X$ (with the covering map being the projection on $X=\Gamma\backslash\mathbb{H})$.
The bijection
\[
	\{n\text{-sheeted covers of }X\}\longleftrightarrow\text{Hom}(\Gamma,S_{n}),
\]
is then given by mapping $\phi$ to $X_{\phi}$.

To see the inverse of this bijection, let $p:\hat{X}\to X$ be an $n$-sheeted
cover of $X$ and label $p^{-1}(o)=\{1,...,n\}$. Let $\gamma\in\pi_{1}(X,o)$, which we consider as a map $\gamma:[0,1]\to X$ for which $\gamma(0)=\gamma(1)=o$,
and let $i\in\{1,...,n\}=p^{-1}(o)$. One can uniquely lift $\gamma$ to a path in
$\hat{X}$ starting at $i$, that is, one can find a map $\hat{\gamma}:[0,1]\to \hat{X}$ for which $\hat{\gamma}(0)=i$ and $\gamma = p\circ\hat{\gamma}$.
As $\gamma(1)=o$ we must have that $\hat{\gamma}(1)\in p^{-1}(o)=\{1,...,n\}$, that is, the endpoint of the lift $\hat{\gamma}$ lies in $\{1,...,n\}$.
We denote this point as $\phi(\gamma)i$.
One also notes that the map $\phi(\gamma)$ from
$\{1,...,n\}$ to itself is actually a permutation, as its inverse
is $\phi(\gamma^{-1})$ where $\gamma^{-1}$ as a loop in $X$ is
just $\gamma$ with reversed orientation. The map $\gamma\mapsto\phi(\gamma)$
is trivially a group homomorphism $\Gamma\to S_{n}$ when $S_{n}$ is equipped with the group structure of composing permutations from left to right.
This is the map:
\[
	\{n\text{-sheeted covers of }X\}\to\text{Hom}(\Gamma,S_{n}),
\]
in the bijection:
\[
	\{n\text{-sheeted covers of }X\}\longleftrightarrow\text{Hom}(\Gamma,S_{n}).
\]

\subsection{The Trace Formula}\label{trace sec}
For $\sigma\in S_{n}$ denote by $\text{Fix}(\sigma)$
the set of its fixed points when acting on $\{1,...,n\}$, we also denote:
\[
	F_{n}(\gamma)\stackrel{\text{def}}{=} \#\text{Fix}(\phi(\gamma)),
\]
for a uniformly random $\phi \in \text{Hom}(\Gamma,S_n)$.
We think of $F_{n}(\gamma)$ as a random variable on the space $\text{Hom}(\Gamma,S_n)$.

Denote by $\mathcal{P}$
the set of primitive conjugacy classes in $\Gamma$ different from the identity.
By primitive we mean that they are conjugacy classes of non-power elements.
We also let $\mathcal{P}_{0}$ be $\mathcal{P}$ where we identify a conjugacy class $[\gamma]$ and its inverse $[\gamma^{-1}]$.
These sets also admit nice topological/geometric interpretations.
The set $\mathcal{P}$ corresponds to primitive oriented geodesics on $X$ while $\mathcal{P}_{0}$ corresponds to primitive non-oriented geodesics on $X$.

Using some abuse of notation, we write $\gamma \in \mathcal{P}$ or $\gamma \in \mathcal{P}_{0}$ to mean that the conjugacy class of $\gamma$
is an element of $\mathcal{P}$ respectively $\mathcal{P}_{0}$. With this notation in mind, note that saying that two elements $\gamma,\delta\in \mathcal{P}_{0}$ are
distinct means that $\gamma$ is not conjugate to $\delta$ or $\delta^{-1}$.
For $\gamma\in \mathcal{P}$ or $\gamma\in \mathcal{P}_{0}$ we denote by $l_{\gamma}$ the length of the associated geodesic on $X$.

Following the above, we introduce an important tool for studying $N_{\phi}(L)$, namely Selberg's trace formula.
We use the twisted version which can be easily derived by combining Theorem 2.2 and Proposition 2.1 of \cite{Naud}:
\begin{thm}[Twisted trace formula]\label{trace} Let $f$ be a real-valued even function on $\mathbb{R}$ whose Fourier transform is compactly supported and smooth. In particular, $f$ extends to an entire function.
	Then for any $n$ and $\phi \in \text{Hom}(\Gamma,S_n)$ we have:
	\[
		N(f;\phi) \stackrel{\text{def}}{=} \sum_{j\geq0}f(r_{\phi,\chi,j}) = N^{\text{det}}(f;\phi) + N^{\text{osc}}(f;\phi),
	\]
	where:
	\[
		N^{\text{det}}(f;\phi) = n(g-1)\int_{\mathbb{R}}f(r)r\tanh(\pi r)dr,
	\]
	and:
	\[
		N^{\text{osc}}(f;\phi)=\sum_{\substack{\gamma\in\mathcal{P}\\k\geq1}}\frac{l_{\gamma}\hat{f}(kl_{\gamma})}{2\sinh(kl_{\gamma}/2)}\chi(\gamma^{k})\#\textnormal{Fix}(\phi(\gamma^{k})).
	\]
\end{thm}
Note that in the term $N^{\text{osc}}(f;\phi)$ we can collect together the terms containing
$\gamma$ and $\gamma^{-1}$ and have the sum be over $\mathcal{P}_{0}$.
Using:
\[
	\chi(\gamma^{k})+\overline{\chi(\gamma^{k})} = 2\mathfrak{R}(\chi(\gamma^{k})),
\]
we get:
\[
	N^{\text{osc}}(f;\phi)=\sum_{\substack{\gamma\in\mathcal{P}_{0}\\k\geq1}}\frac{\mathfrak{R}(\chi(\gamma^{k}))l_{\gamma}\hat{f}(kl_{\gamma})}{\sinh(kl_{\gamma}/2)}\#\text{Fix}(\phi(\gamma^{k})).
\]

Recall that earlier we defined:
\[
	N_{\phi}(L)=\sum_{j\geq 0}h(r_{\phi,\chi,j}),
\]
for:
\[
	h(r)=\psi(L(r-\alpha))+\psi(L(r+\alpha)).
\]
If $\phi$ is chosen uniformly at random from $\text{Hom}(\Gamma,S_n)$, the value of the statistic $N_{\phi}(L)$ is then a random variable on the space $\text{Hom}(\Gamma,S_n)$ which we denote by $N_{n}(L)$. Theorem \ref{trace} shows that this random variable may be decomposed as:
\[
	N_{n}(L) = N_{n}^{\text{det}}(L) + N_{n}^{\text{osc}}(L),
\]
where:
\[
	N_{n}^{\text{det}}(L) = n(g-1)\int_{\mathbb{R}}h(r)r\tanh(\pi r)dr,
\]
is a constant, and:
\[
	N_{n}^{\text{osc}}(L)=\sum_{\substack{\gamma\in\mathcal{P}_{0}\\k\geq1}}\frac{\mathfrak{R}(\chi(\gamma^{k}))l_{\gamma}\hat{h}(kl_{\gamma})}{\sinh(kl_{\gamma}/2)}F_{n}(\gamma^{k}).
\]

\subsection{The Variables $F_{n}(\gamma)$}
The proofs of Theorems \ref{main thm 1} and \ref{almost sure} require a fair bit of knowledge regarding the variables $F_{n}(\gamma)$.
Specifically, we need a way to estimate expressions of the form:
\[
	\mathbb{E}_{n}\left[F_{n}(\gamma^{2})F_{n}(\gamma^{3})F_{n}(\delta^{4})\right],
\]
for distinct $\gamma, \delta \in \mathcal{P}_{0}$ and large $n$.
Recently, Puder, Magee, and Zimhony considered these questions \cite{MP1, PZ}, showing:
\begin{thm}[Corollary 1.7 in \cite{PZ}]
	Let $1\neq\gamma\in \Gamma$ and write $\gamma = \gamma_{0}^q$ for $\gamma_{0}$ primitive and $q$ a positive integer. We have:
	\[
		F_{n}(\gamma) \xrightarrow{\text{dis}} \sum_{d|q}dZ_{1/d},
	\]
	where the $\{Z_{1/d}\}_{d\geq 1}$ are independent Poisson random variables with parameters $1/d$. In fact:
	\[
		\mathbb{E}_{n}[F_{n}(\gamma)] = d(q) + O_{\gamma}(1/n),
	\]
	where $d(q)$ is the number of positive divisors of $q$.
\end{thm}
See the introduction of \cite{surfaces} for a broader overview of the results of Magee, Puder, and Zimhony.

We also have the following theorem and its corollaries from \cite{surfaces} which concern the independence of the variables $F_{n}(\gamma)$:
\begin{thm}[Theorem 1.8 in \cite{surfaces}] \label{main thm 2}
	Let $\gamma_{1},...,\gamma_{t}\in\mathcal{P}_{0}$ be distinct and for each $i$ let $r_{i}\geq 1$ and:
	\[
		a_{i,1},...,a_{i,r_{i}}\geq1,
	\]
	be integers. As $n\to\infty$ we have:

	\[
		\mathbb{E}_{n}\left[\prod_{i=1}^{t}\prod_{j=1}^{r_{i}}F_{n}(\gamma_{i}^{a_{i,j}})\right]=\prod_{i=1}^{t}\mathbb{E}_{n}\left[\prod_{j=1}^{r_{i}}F_{n}(\gamma_{i}^{a_{i,j}})\right] + O(1/n).
	\]
	with the implied constant dependent on  $\gamma_{1},...,\gamma_{t}$ and the integers $a_{i,j}$.
\end{thm}

As a corollary, we get:
\begin{cor}[Corollary 1.9 in \cite{surfaces}] \label{main cor 2}
	In the same setting as Theorem \ref{main thm 2} we have:
	\[
		\mathbb{E}_{n}\left[\prod_{i=1}^{t}\prod_{j=1}^{r_{i}}F_{n}(\gamma_{i}^{a_{i,j}})\right]=\prod_{i=1}^{t}\underset{n\to\infty}{\lim}\mathbb{E}_{n}\left[\prod_{j=1}^{r_{i}}F_{n}(\gamma_{i}^{a_{i,j}})\right] + O(1/n),
	\]
	and:
	\[
		\underset{n\to\infty}{\lim}\mathbb{E}_{n}\left[\prod_{i=1}^{t}\prod_{j=1}^{r_{i}}F_{n}(\gamma_{i}^{a_{i,j}})\right]=\prod_{i=1}^{t}\underset{n\to\infty}{\lim}\mathbb{E}_{n}\left[\prod_{j=1}^{r_{i}}F_{n}(\gamma_{i}^{a_{i,j}})\right].
	\]
\end{cor}

As an additional corollary, we get:
\begin{cor}[Corollary 1.10 in \cite{surfaces}] \label{dist cor}
	Let $\gamma_{1},...,\gamma_{t}\in\mathcal{P}_{0}$
	be distinct and for each $i$ let $r_{i}\geq 1$ and $a_{i,1},...,a_{i,r_{i}}\geq1$
	be integers. For each positive integer $k$ and $1\leq i \leq t$ let $Z_{1/k}^{(i)}$ be
	a Poisson random variable with parameter $1/k$, such that all $Z$-s are independent (in the strong sense, not just pairwise independent).
	Define:
	\[
		X^{(i)}_{a_{i,1},...,a_{i,r_{i}}} = \prod_{j=1}^{r_{i}}\sum_{k|a_{i,j}}kZ_{1/k}^{(i)},
	\]
	and note that for different $i$ the variables $X^{(i)}_{a_{i,1},...,a_{i,r_{i}}}$ are independent.
	Then the cross moments of:
	\[
		\prod_{j=1}^{r_{1}}F_{n}(\gamma_{1}^{a_{1,j}}),...,\prod_{j=1}^{r_{t}}F_{n}(\gamma_{t}^{a_{t,j}}),
	\]
	and of:
	\[
		X^{(1)}_{a_{1,1},...,a_{1,r_{1}}},...,X^{(t)}_{a_{t,1},...,a_{t,r_{t}}},
	\]
	are asymptotically equal. That is, for every $s_{1},...,s_{t}\in \mathbb{N}$ we have:
	\begin{multline}\label{cross eq}
		\underset{n\to\infty}{\lim}\mathbb{E}_{n}\left[\left(\prod_{j=1}^{r_{1}}F_{n}(\gamma_{1}^{a_{1,j}})\right)^{s_{1}}\cdot...\cdot\left(\prod_{j=1}^{r_{t}}F_{n}(\gamma_{t}^{a_{t,j}})\right)^{s_{t}}\right] = \\
		\mathbb{E}\left[\left(X^{(1)}_{a_{1,1},...,a_{1,r_{1}}}\right)^{s_{1}}\cdot...\cdot \left(X^{(t)}_{a_{t,1},...,a_{t,r_{t}}}\right)^{s_{t}}\right] = \\
		\mathbb{E}\left[\left(X^{(1)}_{a_{1,1},...,a_{1,r_{1}}}\right)^{s_{1}}\right] \cdot...\cdot\mathbb{E}\left[\left(X^{(t)}_{a_{t,1},...,a_{t,r_{t}}}\right)^{s_{t}}\right].
	\end{multline}
\end{cor}
Note that each variable $Z_{1/k}^{i}$ may appear several times in (\ref{cross eq}). Furthermore, it is worth mentioning that a stronger result than Corollary \ref{dist cor} is known.
Specifically, the variables $F_{n}(\gamma_{i}^{a_{i,j}})$ jointly converge in distribution towards the variables $X^{(i)}_{a_{i,j}}$, as demonstrated in \cite[Theorems 1.11 and 1.12]{surfaces}.

An example is in order.
Let $\gamma,\delta\in \mathcal{P}_{0}$ be distinct and suppose we wish to estimate $\mathbb{E}_{n}\left[F_{n}(\gamma^{2})F_{n}(\gamma^{3})F_{n}(\delta^{4})\right]$ for large $n$.
Using Corollary \ref{main cor 2} we have:
\begin{multline*}
	\mathbb{E}_{n}\left[F_{n}(\gamma^{2})F_{n}(\gamma^{3})F_{n}(\delta^{4})\right] = \\
	\underset{n\to\infty}{\lim}\mathbb{E}_{n}\left[F_{n}(\gamma^{2})F_{n}(\gamma^{3})\right]\underset{n\to\infty}{\lim}\mathbb{E}_{n}\left[F_{n}(\delta^{4})\right] + O_{\gamma,\delta}(1/n).
\end{multline*}
Using Corollary \ref{dist cor} for $F_{n}(\gamma^{2})F_{n}(\gamma^{3})$ and $F_{n}(\delta^{4})$ we have:
\begin{multline*}
	\mathbb{E}_{n}\left[F_{n}(\gamma^{2})F_{n}(\gamma^{3})\right]\to \mathbb{E}\left[X^{(1)}_{2,3}\right] =
	\mathbb{E}\left[\left(Z^{(1)}_{1}+2Z^{(1)}_{1/2}\right)\left(Z^{(1)}_{1} + 3Z^{(1)}_{1/3}\right)\right] =\\
	\mathbb{E}\left[\left(Z^{(1)}\right)^{2} + 2Z^{(1)}_{1}Z^{(1)}_{1/2} + 3Z^{(1)}_{1}Z^{(1)}_{1/3} + 6Z^{(1)}_{1/2}Z^{(1)}_{1/3}\right],
\end{multline*}
and:
\[
	\mathbb{E}_{n}\left[F_{n}(\delta^{4})\right]\to \mathbb{E}\left[X^{(2)}_{4}\right] = \mathbb{E}\left[Z^{(2)}_{1}+2Z^{(2)}_{1/2}+4Z^{(2)}_{1/4}\right].
\]
Recalling that $\mathbb{E}\left[Z_{\lambda}^{2}\right] = \lambda^{2} + \lambda$ we get:
\begin{multline*}
	\underset{n\to\infty}{\lim}\mathbb{E}_{n}\left[F_{n}(\gamma^{2})F_{n}(\gamma^{3})\right]\underset{n\to\infty}{\lim}\mathbb{E}_{n}\left[F_{n}(\delta^{4})\right] =\\
	\left(1+1 + 2\cdot1\cdot\frac{1}{2} + 3\cdot1\cdot\frac{1}{3} + 6\cdot\frac{1}{2}\cdot\frac{1}{3}\right)\cdot\left(1 + 2\cdot\frac{1}{2} + 4\cdot\frac{1}{4}\right) =15.
\end{multline*}
So that:
\[
	\mathbb{E}_{n}\left[F_{n}(\gamma^{2})F_{n}(\gamma^{3})F_{n}(\delta^{4})\right] = 15 + O_{\gamma,\delta}(1/n).
\]

For products of the form $F_{n}(\gamma^{a})F_{n}(\gamma^{b})$, we use the following simple lemma concerning their covariance:
\begin{lem}[Proposition 3.1 in \cite{Naud}]\label{G lem} For $\gamma \in \mathcal{P}_{0}$ and $a,c\geq1$ we have:
	\[
		\underset{n\to\infty}{\lim}\mathbb{E}_{n}\left[\left(F_{n}(\gamma^{a})-\mathbb{E}_{n}[F_{n}(\gamma^{a})]\right)\left(F_{n}(\gamma^{b})-\mathbb{E}_{n}[F_{n}(\gamma^{b})]\right)\right] = G(a,b),
	\]
	where:
	\[
		G(a,b) = \sum_{d|a\&d|b}d = \sigma(\gcd(a,b)),
	\]
	and $\sigma(n) \stackrel{\text{def}}{=} \sum_{d\mid n}d$ is the sum of divisors function.
\end{lem}
Note that $G(a,b)\leq \min\{\sigma(a),\sigma(b)\}$, and as $\sigma(n)\ll n\log(n)$ we have \newline $G(a,b) \ll a\log(a)$.
Also, note that the bound $\sigma(n)\ll n\log n$ is not optimal, it is known that $\sigma(n)\ll n\log\log n$ (see \cite[Theorem 323]{HW}), however, we still use the weaker bound for ease of analysis.

\subsection{Counting Bounds}
As one sees in the trace formula, the lengths of primitive geodesics
on $X$ play a big role in the oscillating term $N_{n}^{\text{osc}}(L)$. An important
counting bound related to these lengths is the so-called ``prime
geodesic theorem'' which is an analog of the prime number theorem. Similarly to the
prime number theorem, one could ask how many closed geodesics are there on $X$ with
length smaller than some $T\gg0$, i.e. to estimate the function:
\par
\[
	N(T)=\sum_{\substack{kl_{\gamma}\leq T}}1,
\]
where $\gamma$ ranges over primitive conjugacy classes in $\Gamma$, i.e. over $\mathcal{P}$, and $k$ ranges over the positive integers.
\par
One could also consider the number of \textbf{primitive} closed geodesics on
$X$ with length smaller than $T$ i.e. to estimate:
\par
\[
	N^{0}(T)=\sum_{l_{\gamma}\leq T}1.
\]
This was first considered by Huber \cite{HB} and later improved by Hejhal \cite{HJ} and Randol \cite{RD}. We state the theorem as it appears in \cite{PB}.
\par
\begin{thm}[Theorem 9.6.1 in \cite{PB}] \label{prime geo} Let $0<\lambda_{1},...,\lambda_{k}\leq 1/4$ be the small eigenvalues of the non-twisted Laplacian on $X$. For each $1\leq i \leq k$ define
	$s_{i} = 1/2+(1/4-\lambda_{i})^{1/2}$, then:
	\[
		N^{0}(T) = \Li(e^T) + \sum_{1>s_{i}>3/4}\Li(e^{s_{i}T}) + O_{X}(e^{\frac{3}{4}T}/T).
	\]
	In particular, there exists a $3/4 \leq \nu < 1$ (dependent on $X$) such that:
	\[
		N^{0}(T)=\Li(e^T) + O_{X}(e^{\nu T}/T).
	\]
\end{thm}
Note that in the proof of Theorem \ref{main thm 1} we will mostly use the following weaker form of the theorem above:
\[
	N^{0}(T)=\frac{e^{T}}{T}(1+o(1)).
\]
Noticing that $N(T)=N^{0}(T)+\sum_{k\geq2}N^{0}(T/k)$ we also get:
\[
	N(T)=N^{0}(T)+O(Te^{T/2}).
\]

For the proof of Theorem \ref{almost sure} we need a bound on the altered counting function:
\[
	N_{\chi}(T) = \sum_{\substack{kl_{\gamma}\leq T\\k\geq1}}\chi(\gamma^{k}),
\]
where $\gamma$ ranges over primitive conjugacy classes in $\Gamma$, i.e. over $\mathcal{P}$.
Proposition 4.2 of \cite{Naud} gives the following:
\begin{prop}\label{prime chi geo}
	For $\chi \neq 1$ there exists a $\beta < 1$ that is dependent on $X,\chi$ such that for $T\to \infty$:
	\[
		\left|N_{\chi}(T)\right| \ll e^{\beta T},
	\]
	where the implied constant depends on $X$.
\end{prop}

\section{Proof of Theorem ~\ref{main thm 1}} \label{main thm 1 proof}
\subsection{Outline of the Proof}
Our proof is fairly standard. We start by using the trace formula in Theorem \ref{trace} to express $\mathbb{V}_{n}^{(k)}(L)$ as a sum over $k$-tuples of elements in $\mathcal{P}_{0}$ and $\mathbb{N}$. We then break apart the sum into smaller sums depending on how the $k$-tuple of elements from $\mathcal{P}_{0}$ are partitioned. Theorem \ref{B bound}, which will be proven in Section \ref{B bound proof}, allows us to bound the smaller sums using information contained in the partition alone. Subsequently, we split into cases based on the parity of $k$.
We show that in the case where $k$ is odd all the smaller sums vanish in the double limit, and in the case that $k$ is even we recover our desired moments.

\begin{proof}[Proof of Theorem ~\ref{main thm 1}]
	Recall that for a given $\phi \in \text{Hom}(\Gamma,S_{n})$ we define:
	\[
		N_{\phi}(L)=\sum_{j\geq 0}h(r_{\phi,\chi,j}),
	\]
	for:
	\[
		h(r)=\psi(L(r-\alpha))+\psi(L(r+\alpha)).
	\]
	Also recall that when $\phi$ is chosen uniformly at random from $\text{Hom}(\Gamma,S_{n})$, we denote the random value of the statistic $N_{\phi}(L)$ as $N_{n}(L)$.
	Using the twisted trace formula, Theorem \ref{trace}, the random variable $N_{n}(L)$ may be decomposed as:
	\[
		N_{n}(L) =N_{n}^{\text{det}}(L) + N_{n}^{\text{osc}}(L),
	\]
	where:
	\[
		N_{n}^{\text{det}}(L) = n(g-1)\int_{\mathbb{R}}h(r)r\tanh(\pi r)dr,
	\]
	and:
	\[
		N_{n}^{\text{osc}}(L) = \sum_{\substack{\gamma\in\mathcal{P}_{0}\\a\geq1}}\frac{\mathfrak{R}(\chi(\gamma^{a}))l_{\gamma}\hat{h}(al_{\gamma})}{\sinh(al_{\gamma}/2)}F_{n}(\gamma^{a}).
	\]

	A simple calculation shows:
	\[
		\hat{h}(\zeta)=\frac{2\cos(\alpha\zeta)}{L}\hat{\psi}(\zeta/L),
	\]
	so that:
	\[
		N_{n}^{\text{osc}}(L)=\frac{2}{L}\sum_{\substack{\gamma\in\mathcal{P}_{0}\\a\geq1}}\frac{\mathfrak{R}(\chi(\gamma^{a}))l_{\gamma}\hat{\psi}(al_{\gamma}/L)\cos(\alpha al_{\gamma})}{\sinh(al_{\gamma}/2)}F_{n}(\gamma^{a}).
	\]
	Set:
	\[
		U_{n}(\gamma)=F_{n}(\gamma)-\mathbb{E}_{n}[F_{n}(\gamma)],
	\]
	and:
	\[
		\begin{split}
			T_{n}(L) & =N_{n}^{\text{osc}}(L)-\mathbb{E}_{n}[N_{n}^{\text{osc}}(L)] \\
			         & =\frac{2}{L}\sum_{\substack{\gamma\in\mathcal{P}_{0}         \\a\geq1}}\frac{\mathfrak{R}(\chi(\gamma^{a}))l_{\gamma}\hat{\psi}(al_{\gamma}/L)\cos(\alpha al_{\gamma})}{\sinh(al_{\gamma}/2)}U_{n}(\gamma^{a}),
		\end{split}
	\]
	so that:
	\[
		\mathbb{V}_{n}^{(k)}(L)=\mathbb{E}_{n}\left[T_{n}(L)^{k}\right].
	\]

	Denote:
	\begin{equation}\label{H def}
		H(\gamma_{1},...,\gamma_{k},a_{1},...,a_{k})=
		\prod_{i=1}^{k}\frac{\mathfrak{R}(\chi(\gamma_{i}^{a_{i}}))l_{\gamma_{i}}\hat{\psi}(a_{i}l_{\gamma_{i}}/L)\cos(\alpha a_{i}l_{\gamma_{i}})}{\sinh(a_{i}l_{\gamma_{i}}/2)},
	\end{equation}
	therefore:
	\begin{multline}\label{original V}
		\mathbb{V}_{n}^{(k)}(L) = \mathbb{E}_{n}\left[T_{n}(L)^{k}\right]=\\
		\frac{2^{k}}{L^{k}}\sum_{\substack{\gamma_{1},...,\gamma_{k}\in\mathcal{P}_{0}\\a_{1},...,a_{k}\geq1}}
		H(\gamma_{1},...,\gamma_{k},a_{1},...,a_{k})\mathbb{E}_{n}\left[U_{n}(\gamma_{1}^{a_{1}})...U_{n}(\gamma_{k}^{a_{k}})\right].
	\end{multline}
	Note that for $\gamma_{1},...,\gamma_{k}\in\mathcal{P}_{0}$ if there exists an
	$i$ for which $l_{\gamma_{i}} \gg L$ then $\hat{\psi}(a_{i}l_{\gamma_{i}}/L) = 0$ (as $\hat{\psi}$ is compactly supported) so that $H(\gamma_{1},...,\gamma_{k},a_{1},...,a_{k}) = 0$ and so the sum above is in fact finite

	Before continuing, let us give a few important definitions.
	\begin{defn}[Partitions]
		Let $k$ be a positive integer. A $t$-tuple with $1\leq t\leq k$ of positive integers $\underline{r} = (r_{1},...,r_{t})$ is called a partition of $k$ if the following criteria are met:
		\begin{itemize}
			\item
			      The $r_{i}$-s are descending, that is $r_{1}\geq r_{2}\geq ...\geq r_{t}$.
			\item
			      $\sum_{i=1}^{t} r_{i}=k$.
		\end{itemize}
		If $\underline{r}$ is a partition of $k$ we denote it by $\underline{r}\vdash k$.
	\end{defn}
	\begin{defn}[Symmetries of a partition]
		For a partition $\underline{r}\vdash k$ we let $\text{Sym}(\underline{r})$ denote the symmetries of $\underline{r}$ that leave it a valid partition of $k$. Formally, for $\underline{r}=(r_{1},...,r_{t})\vdash k$ the group $\text{Sym}(\underline{r})$ is the subgroup of $S_{t}$ such that for every $\tau\in\text{Sym}(\underline{r})$ the $t$-tuple $(r_{\tau(1)},...,r_{\tau(t)})$ is still a partition of $k$. In other words, $\tau \in S_{t}$ is an element of $\text{Sym}(\underline{r})$ if and only if for every $1\leq i\leq j\leq t$ we have $r_{\tau(i)}\geq r_{\tau(j)}$.
	\end{defn}
	\begin{defn}[The set $\mathcal{P}_{0}(\underline{r})$]
		Let $\underline{r}\vdash k$ be a partition. Denote by \newline $\mathcal{P}_{0}(\underline{r}) = \mathcal{P}_{0}(r_{1},...,r_{t})$ the following set:
		\[
			\mathcal{P}_{0}(r_{1},...,r_{t}) \stackrel{def}{=}\left\{\underline{\gamma} = (\overbrace{\gamma_{1},...,\gamma_{1}}^{r_{1}},...,\overbrace{\gamma_{t},...,\gamma_{t}}^{r_{t}}) :\text{all } \gamma_{i}\in\mathcal{P}_{0} \text{ distinct}\right\}.
		\]
	\end{defn}

	Returning to our analysis, a simple manipulation shows that:
	\[
		\mathbb{E}_{n}\left[T_{n}(L)^{k}\right]=\frac{2^{k}}{L^{k}}\sum_{\underline{r}\vdash k}\frac{1}{\#\text{Sym}(\underline{r})}\binom{k}{\underline{r}}B(\underline{r}),
	\]
	where for $\underline{r}=(r_{1},...,r_{t})\vdash k$:
	\[
		\binom{k}{\underline{r}} = \binom{k}{r_{1},...,r_{t}} = \frac{k!}{r_{1}!...r_{t}!},
	\]
	is the usual multinomial coefficient,
	and $B(\underline{r})$ is defined as:
	\[
		B(\underline{r}) = B(r_{1},...,r_{t}) = \sum_{\substack{
				\underline{\gamma}\in\mathcal{P}_{0}(\underline{r})\\
				\underline{a}\in \mathbb{N}^{k}
			}}H(\underline{\gamma},\underline{a})\mathbb{E}_{n}[U_{n}(\underline{\gamma}^{\underline{a}})],
	\]
	where for $\underline{\gamma}\in \mathcal{P}_{0}(\underline{r})$ and $\underline{a}=(a_{1,1},...,a_{1,r_{1}},...,a_{t,1},...,a_{t,r_{t}})\in \mathbb{N}^{k}$
	we define:
	\[
		U_{n}(\underline{\gamma}^{\underline{a}})=\prod_{i=1}^{t}\prod_{j=1}^{r_{i}}U_{n}(\gamma_{i}^{a_{i,j}}).
	\]
	All in all:
	\begin{equation}\label{V n k eq}
		\mathbb{V}_{n}^{(k)}(L)=\mathbb{E}_{n}\left[T_{n}(L)^{k}\right] = \frac{2^{k}}{L^{k}}\sum_{\underline{r}\vdash k}\frac{1}{\#\text{Sym}(\underline{r})}\binom{k}{\underline{r}}B(\underline{r}).
	\end{equation}

	The core of our proof relies on the following result which allows us to bound the value of $B(\underline{r})$ given $\underline{r}$.
	\begin{thm}\label{B bound}
		Let $k$ be a positive integer and let $\underline{r} = (r_{1},...,r_{t}) \vdash k$ be a partition of $k$.
		We have the following bound on $B(\underline{r}) = B(r_{1},...,r_{t})$:
		\begin{itemize}
			\item
			      If there exists an $i$ for which $r_{i}=1$ then:
			      \[
				      \underset{n\to\infty}{\lim}B(\underline{r}) = 0.
			      \]
			\item
			      If for all $i$ we have $r_{i}\geq 2$ then:
			      \[
				      \underset{n\to\infty}{\lim}B(\underline{r}) \ll L^{2\#\{r_{i}=2\}}.
			      \]
		\end{itemize}
	\end{thm}

	Given the above theorem, let us prove Theorem \ref{main thm 1}. We split into the $k$-odd and $k$-even cases:
	\subsection{Odd $k$}
	Let $k$ and $t$ be as in Theorem \ref{B bound} with $k$ odd. As $k$ is odd, for every partition $\underline{r}=(r_{1},...,r_{t})\vdash k$ either
	one of the $r_{i}$-s is equal to $1$, or one of the $r_{i}$-s is at least $3$.
	If one of the $r_{i}$-s is equal to $1$ then by the first bulletin of Theorem \ref{B bound}:
	\[
		\underset{n\to\infty}{\lim}\frac{2^{k}}{L^{k}}\frac{1}{\#\text{Sym}(\underline{r})}\binom{k}{\underline{r}}B(\underline{r}) = 0.
	\]
	On the other hand, if one of the $r_{i}$-s is at least $3$ then:
	\[
		2\#\{r_{i}=2\} \leq k - \sum_{r_{i}\geq3}r_{i}\leq k-3,
	\]
	and the second bulletin of Theorem \ref{B bound} gives:
	\[
		\underset{n\to\infty}{\lim}\frac{2^{k}}{L^{k}}\frac{1}{\#\text{Sym}(\underline{r})}\binom{k}{\underline{r}}B(\underline{r}) \ll L^{-3},
	\]
	so that:
	\[
		\underset{L\to\infty}{\lim}\underset{n\to\infty}{\lim}\frac{2^{k}}{L^{k}}\frac{1}{\#\text{Sym}(\underline{r})}\binom{k}{\underline{r}}B(\underline{r}) = 0.
	\]

	These two observations, together with Equation \ref{V n k eq}, yield:
	\[
		\begin{split}
			\underset{L\to\infty}{\lim}\underset{n\to\infty}{\lim}\mathbb{V}_{n}^{(k)}(L) & =
			\underset{L\to\infty}{\lim}\underset{n\to\infty}{\lim}\mathbb{E}_{n}[T_{n}(L)^{k}]                                                                                                                                                                                  \\
			                                                                              & =\underset{L\to\infty}{\lim}\underset{n\to\infty}{\lim}\frac{2^{k}}{L^{k}}\sum_{\underline{r}\vdash k}\frac{1}{\#\text{Sym}(\underline{r})}\binom{k}{\underline{r}}B(\underline{r}) \\
			                                                                              & =\sum_{\underline{r}\vdash k}\underset{L\to\infty}{\lim}\underset{n\to\infty}{\lim}\frac{2^{k}}{L^{k}}\frac{1}{\#\text{Sym}(\underline{r})}\binom{k}{\underline{r}}B(\underline{r}) \\
			                                                                              & =\sum_{\underline{r}\vdash k} 0                                                                                                                                                     \\
			                                                                              & = 0.
		\end{split}
	\]
	Note that we can exchange the summation and the limit as the sum over $\underline{r} \vdash k$ is finite (there are only a finite number of partitions of $k$).

	\subsection{Even $k$}
	Reasoning similar to the case where $k$ is odd shows that for even $k$ the only partitions $\underline{r} = (r_{1},...,r_{t}) \vdash k$ for which
	\[
		\underset{L\to\infty}{\lim}\underset{n\to\infty}{\lim}\frac{2^{k}}{L^{k}}\frac{1}{\#\text{Sym}(\underline{r})}\binom{k}{\underline{r}}B(\underline{r}) \neq 0,
	\]
	are the ones where $t=\frac{k}{2}$ and $r_{1}=...=r_{\frac{k}{2}}=2$. We denote this special partition by $2^{(k)}$ so that:
	\[
		B(2^{(k)}) =
		\sum_{\substack{
		\underline{\gamma}\in\mathcal{P}_{0}(2^{(k)})\\
		\underline{a}=(a_{1},b_{1},...,a_{k/2},b_{k/2})\\a_{i},b_{i}\geq1
		}}H(\underline{\gamma},\underline{a})\mathbb{E}_{n}[U_{n}(\underline{\gamma}^{\underline{a}})].
	\]
	Note that although the summation might seem infinite, the sum above is finite. Recalling the definition of $H(\underline{\gamma},\underline{a})$ from Equation \ref{H def}:
	\[
		H(\gamma_{1},...,\gamma_{k},a_{1},...,a_{k})=
		\prod_{i=1}^{k}\frac{\mathfrak{R}(\chi(\gamma_{i}^{a_{i}}))l_{\gamma_{i}}\hat{\psi}(a_{i}l_{\gamma_{i}}/L)\cos(\alpha a_{i}l_{\gamma_{i}})}{\sinh(a_{i}l_{\gamma_{i}}/2)},
	\]
	one notices that due to the compact support of $\hat{\psi}$, the expression $H(\underline{\gamma},\underline{a})$ vanishes whenever there is an $a_{i} \gg L$ or a $\gamma_{i}$ with length $l_{\gamma_{i}} \gg L$. Thus the sum defining $B(2^{(k)})$  is in fact finite.

	Let $\gamma_{1},...,\gamma_{k/2}\in \mathcal{P}_{0}$ be distinct and let $a_{1},b_{1},...,a_{k/2},b_{k/2}$ be positive integers.
	For:
	\[
		\underline{\gamma} = (\gamma_{1},\gamma_{1},...,\gamma_{k/2},\gamma_{k/2})\in\mathcal{P}_{0}(2^{(k)}),
	\]
	and $\underline{a} = (a_{1},b_{1},...,a_{k/2},b_{k/2})$,
	notice that:
	\begin{multline*}
		H(\underline{\gamma},\underline{a})=H(\gamma_{1},\gamma_{1},\gamma_{2},\gamma_{2},...,\gamma_{k/2},a_{1},b_{1},...,a_{k/2},b_{k/2}) =\\
		\prod_{i=1}^{k/2}H(\gamma_{i},\gamma_{i},a_{i},b_{i}).
	\end{multline*}
	In addition, Corollary \ref{main cor 2} gives:
	\begin{multline*}
		\underset{n\to\infty}{\lim}\mathbb{E}_{n}\left[U_{n}(\underline{\gamma}^{\underline{a}})\right]=
		\underset{n\to\infty}{\lim}\mathbb{E}_{n}\left[\prod_{i=1}^{k/2}U_{n}(\gamma_{i}^{a_{i}})U_{n}(\gamma_{i}^{b_{i}})\right] =\\
		\prod_{i=1}^{k/2}\underset{n\to\infty}{\lim}\mathbb{E}_{n}\left[U_{n}(\gamma_{i}^{a_{i}})U_{n}(\gamma_{i}^{b_{i}})\right].
	\end{multline*}
	As a consequence:
	\begin{multline}\label{B sum}
		\underset{n\to\infty}{\lim}B(2^{(k)}) =
		\sum_{\substack{
		\underline{\gamma}\in\mathcal{P}_{0}(2^{(k)})\\
		\underline{a}=(a_{1},b_{1},...,a_{k/2},b_{k/2})\\a_{i},b_{i}\geq1
		}}\underset{n\to\infty}{\lim}H(\underline{\gamma},\underline{a})\mathbb{E}_{n}[U_{n}(\underline{\gamma}^{\underline{a}})]\\
		= \sum_{\substack{
		\underline{\gamma}\in\mathcal{P}_{0}(2^{(k)})\\
		\underline{a}=(a_{1},b_{1},...,a_{k/2},b_{k/2})\\a_{i},b_{i}\geq1
		}}\prod_{i=1}^{k/2}H(\gamma_{i},\gamma_{i},a_{i},b_{i})\underset{n\to\infty}{\lim}\mathbb{E}_{n}\left[U_{n}(\gamma_{i}^{a_{i}})U_{n}(\gamma_{i}^{b_{i}})\right].
	\end{multline}
	Note that we can exchange the sum and the limit as the sum defining $B\left(2^{(k)}\right)$ is finite.

	Next, we wish to rewrite the sum in the LHS of Equation \ref{B sum} as a product. If the sum were over all possible $\gamma_{1}, ..., \gamma_{k/2} \in \mathcal{P}_{0}$ then it would equal:
	\begin{equation}\label{H prod}
		\prod_{i=1}^{k/2}\left[\sum_{\substack{\gamma\in\mathcal{P}_{0}\\a,b\geq1}}H(\gamma,\gamma,a,b)\underset{n\to\infty}{\lim}\mathbb{E}_{n}\left[U_{n}(\gamma^{a})U_{n}(\gamma^{b})\right]\right].
	\end{equation}
	However, as the sum is over $\underline{\gamma}\in\mathcal{P}_{0}$ for:
	\[
		\underline{\gamma} =(\gamma_{1},\gamma_{1},...,\gamma_{k/2},\gamma_{k/2}),
	\]
	all $\gamma_{i}$-s are distinct. Note that when opening up the product in Equation \ref{H prod} we get a sum over all possible tuples $\underline{\gamma}$ for which all $\gamma_{i}$-s are not necessarily distinct. Thus, to make the product in Equation \ref{H prod} equal to the sum in the LHS of Equation \ref{B sum} we must subtract from it the contribution from the terms where $\#\{\gamma_{1},...,\gamma_{k/2}\} < k/2$.

	For each choice of $\gamma_{1}, ..., \gamma_{k/2}$ such that $\#\{\gamma_{1},...,\gamma_{k/2}\} < k/2$, we look at the partition of $k/2$ defined by the distinct elements of $(\gamma_{1}, .., \gamma_{k/2})$. For example, in the $k=6$ case if $\gamma, \delta, \epsilon$ are distinct elements of $\mathcal{P}_{0}$ then the tuples $(\gamma,\gamma,\delta), (\delta, \epsilon, \delta)$ define the partition $(2,1)$ of $3$ while the tuples $(\gamma, \delta, \epsilon), (\epsilon, \gamma, \delta)$  define the partition $(1,1,1)$ of $3$.

	When subtracting from the product in Equation \ref{H prod} the contribution from the terms where $\#\{\gamma_{1},...,\gamma_{k/2}\} < k/2$, we group these terms by the partition of $k/2$ they define. For each such partition $\underline{r} \vdash k/2$
	there are only $\frac{1}{\#\text{Sym}(\underline{r})}\binom{k/2}{\underline{r}}$ ways it can occur in a $k/2$-tuple, we call each such way a $k/2$-tuple template of $\underline{r}$. For example, in the $k=6$ with $\gamma, \delta, \epsilon$ as before, the tuples $(\gamma,\gamma,\delta)$ and $(\delta,\delta,\epsilon)$ define the same $3$-tuple template of the partition $(2,1)$, while $(\gamma,\delta,\gamma)$ defines a different template.

	The sum over $(\gamma_{1}, .., \gamma_{k/2})$ partitioned according to any given $k/2$-tuple template of $\underline{r}$ is $B(2\underline{r})$, where for a partition $\underline{r}=(r_{1},...,r_{s})\vdash k/2$ we denote by $2\underline{r}$ the partition $(2r_{1},...,2r_{s})$ of $k$.  As a result, we get the following equality:

	\begin{multline}\label{prod split}
		\sum_{\substack{
		\underline{\gamma}\in\mathcal{P}_{0}(2^{(k)})\\
		\underline{a}=(a_{1},b_{1},...,a_{k/2},b_{k/2})\\a_{i},b_{i}\geq1
		}}\prod_{i=1}^{k/2}H(\gamma_{i},\gamma_{i},a_{i},b_{i})\underset{n\to\infty}{\lim}\mathbb{E}_{n}\left[U_{n}(\gamma_{i}^{a_{i}}\gamma_{i}^{b_{i}})\right] = \\
		\prod_{i=1}^{k/2}\left[\sum_{\substack{\gamma\in\mathcal{P}_{0}\\a,b\geq1}}H(\gamma,\gamma,a,b)\underset{n\to\infty}{\lim}\mathbb{E}_{n}\left[U_{n}(\gamma^{a})U_{n}(\gamma^{b})\right]\right] - \\
		\sum_{\substack{\underline{r} = (r_{1},...,r_{s})\vdash k/2 \\ 1\leq s <k/2}}\frac{1}{\#\text{Sym}(\underline{r})}\binom{k/2}{\underline{r}}\underset{n\to\infty}{\lim}B(2\underline{r}).
	\end{multline}

	For $1\leq s<k/2$ every partition:
	\[
		\underline{r}=(r_{1},...,r_{s})\vdash k/2,
	\]
	has an $i$ for which $r_{i}\geq 2$. Theorem \ref{B bound} gives:
	\[
		\sum_{\substack{\underline{r} = (r_{1},...,r_{s})\vdash k/2 \\ 1\leq s <k/2}}\frac{1}{\#\text{Sym}(\underline{r})}\binom{k/2}{\underline{r}}\underset{n\to\infty}{\lim}B(2\underline{r}) \ll L^{k-4},
	\]
	with the implied constant dependent on $k$.

	Let us now evaluate the product in the RHS of Equation \ref{prod split}. Note that its terms resemble the terms of the sum in Equation \ref{original V}. Corollary \ref{main cor 2}, implies that:
	\[
		\underset{n\to\infty}{\lim} \mathbb{E}_{n}\left[U_{n}(\gamma^{a})U_{n}(\delta^{b})\right] = 0,
	\]
	whenever $\gamma,\delta \in \mathcal{P}_{0}$ are distinct. Using Equation \ref{original V} we get:
	\begin{multline*}
		\sum_{\substack{\gamma\in\mathcal{P}_{0}\\a,b\geq1}}H(\gamma,\gamma,a,b)\underset{n\to\infty}{\lim}\mathbb{E}_{n}\left[U_{n}(\gamma^{a})U_{n}(\gamma^{b})\right] =\\
		\underset{n\to\infty}{\lim}\sum_{\substack{\gamma,\delta\in\mathcal{P}_{0}\\a,b\geq1}}H(\gamma,\delta,a,b)\mathbb{E}_{n}\left[U_{n}(\gamma^{a})U_{n}(\delta^{b})\right] =
		\underset{n\to\infty}{\lim}\frac{L^{2}}{4}\mathbb{V}_{n}^{(2)}(L).
	\end{multline*}
	Note that we are allowed to exchange the sum and the limit as the sum is finite.
	All in all, combining the above with Equation \ref{prod split} we get:
	\begin{multline}\label{most B 2}
		\underset{n\to\infty}{\lim}B(2^{(k)}) =\sum_{\substack{
		\underline{\gamma}\in\mathcal{P}_{0}(2^{(k)})\\
		\underline{a}=(a_{1},b_{1},...,a_{k/2},b_{k/2})\\a_{i},b_{i}\geq1
		}}\prod_{i=1}^{k/2}H(\gamma_{i},\gamma_{i},a_{i},b_{i})\underset{n\to\infty}{\lim}\mathbb{E}_{n}\left[U_{n}(\gamma_{i}^{a_{i}}\gamma_{i}^{b_{i}})\right] = \\
		\prod_{i=1}^{k/2}\left[\sum_{\substack{\gamma\in\mathcal{P}_{0}\\a,b\geq1}}H(\gamma,\gamma,a,b)\underset{n\to\infty}{\lim}\mathbb{E}_{n}\left[U_{n}(\gamma^{a})U_{n}(\gamma^{b})\right]\right] - \\
		\sum_{\substack{\underline{r} = (r_{1},...,r_{s})\vdash k/2 \\ 1\leq s <k/2}}\frac{1}{\#\text{Sym}(\underline{r})}\binom{k/2}{\underline{r}}\underset{n\to\infty}{\lim}B(2\underline{r}) = \\
		\left[\underset{n\to\infty}{\lim}\frac{L^{2}}{4}\mathbb{V}_{n}^{(2)}(L)\right]^{k/2} - O\left(L^{k-4}\right) =
		\frac{L^{k}}{2^{k}}\left[\underset{n\to\infty}{\lim}\mathbb{V}_{n}^{(2)}(L)\right]^{k/2} - O(L^{k-4}).
	\end{multline}

	Note that:
	\begin{multline*}
		\frac{2^{k}}{L^{k}}\sum_{\underline{r}\vdash k}\frac{1}{\#\text{Sym}(\underline{r})}\binom{k}{\underline{r}}B(\underline{r})=\\
		\frac{2^{k}}{L^{k}}\frac{1}{\#\text{Sym}(2^{(k)})}\binom{k}{2^{(k)}}B(2^{(k)}) +
		\sum_{\substack{\underline{r}=(r_{1},...,r_{t})\vdash k \\ \exists r_{i}\neq 2}}\frac{2^{k}}{L^{k}}\frac{1}{\#\text{Sym}(\underline{r})}\binom{k}{\underline{r}}B(\underline{r}).
	\end{multline*}
	Theorem \ref{B bound} alongside the fact that there are only a finite number of partitions of $k$ gives:
	\begin{multline*}
		\underset{L\to\infty}{\lim}\underset{n\to\infty}{\lim}\sum_{\substack{\underline{r}=(r_{1},...,r_{t})\vdash k \\ \exists r_{i}\neq 2}}\frac{2^{k}}{L^{k}}\frac{1}{\#\text{Sym}(\underline{r})}\binom{k}{\underline{r}}B(\underline{r}) =\\
		\sum_{\substack{\underline{r}=(r_{1},...,r_{t})\vdash k \\ \exists r_{i}\neq 2}}\underset{L\to\infty}{\lim}\underset{n\to\infty}{\lim}\frac{2^{k}}{L^{k}}\frac{1}{\#\text{Sym}(\underline{r})}\binom{k}{\underline{r}}B(\underline{r})
		=\sum_{\substack{\underline{r}=(r_{1},...,r_{t})\vdash k \\ \exists r_{i}\neq 2}}0= 0.
	\end{multline*}
	Thus:
	\[
		\begin{split}
			\underset{L\to\infty}{\lim}\underset{n\to\infty}{\lim}\mathbb{E}_{n}\left[T_{n}(L)^{k}\right] & =
			\underset{L\to\infty}{\lim}\underset{n\to\infty}{\lim}\frac{2^{k}}{L^{k}}\sum_{\underline{r}\vdash k}\frac{1}{\#\text{Sym}(\underline{r})}\binom{k}{\underline{r}}B(\underline{r})                                                     \\
			                                                                                              & =\underset{L\to\infty}{\lim}\underset{n\to\infty}{\lim}\frac{2^{k}}{L^{k}}\frac{1}{\#\text{Sym}(2^{(k)})}\binom{k}{2^{(k)}}B(2^{(k)}).
		\end{split}
	\]

	Recalling from Equation \ref{most B 2} that:
	\[
		\underset{n\to\infty}{\lim}B(2^{(k)}) = \frac{L^{k}}{2^{k}}\left[\underset{n\to\infty}{\lim}\mathbb{V}_{n}^{(2)}(L)\right]^{k/2} - O(L^{k-4}),
	\]
	we get:
	\begin{multline*}
		\underset{L\to\infty}{\lim}\underset{n\to\infty}{\lim}\mathbb{E}_{n}\left[T_{n}(L)^{k}\right]  =
		\underset{L\to\infty}{\lim}\underset{n\to\infty}{\lim}\frac{2^{k}}{L^{k}}\frac{1}{\#\text{Sym}(2^{(k)})}\binom{k}{2^{(k)}}B(2^{(k)})=\\
		\underset{L\to\infty}{\lim}\frac{2^{k}}{L^{k}}\frac{1}{\#\text{Sym}(2^{(k)})}\binom{k}{2^{(k)}}\left(\frac{L^{k}}{2^{k}}\left[\underset{n\to\infty}{\lim}\mathbb{V}_{n}^{(2)}(L)\right]^{k/2} - O(L^{k-4})\right)=\\
		\frac{1}{\#\text{Sym}(2^{(k)})}\binom{k}{2^{(k)}}\left[\underset{L\to\infty}{\lim}\underset{n\to\infty}{\lim}\mathbb{V}_{n}^{(2)}(L)\right]^{k/2}.
	\end{multline*}

	The partition $2^{(k)}$ has maximal symmetry, that is $\text{Sym}(2^{(k)}) = S_{k/2}$. In particular $\#\text{Sym}(2^{(k)}) = (k/2)!$.
	As $\binom{k}{2^{(k)}} = \frac{k!}{2^{k/2}}$, the standard relation:
	\[
		(k-1)!! = \frac{k!}{2^{k/2}(k/2)!},
	\]
	yields:
	\[
		\frac{1}{\#\text{Sym}(2^{(k)})}\binom{k}{2^{(k)}} = (k-1)!!.
	\]
	Using Theorem \ref{Naud main} we have:
	\[
		\begin{split}
			\underset{L\to\infty}{\lim}\underset{n\to\infty}{\lim}\mathbb{V}_{n}^{(k)}(L) & =
			\underset{L\to\infty}{\lim}\underset{n\to\infty}{\lim}\mathbb{E}_{n}\left[T_{n}(L)^{k}\right]                                                                                                                                      \\
			                                                                              & =\frac{1}{\#\text{Sym}(2^{(k)})}\binom{k}{2^{(k)}}\left[\underset{L\to\infty}{\lim}\underset{n\to\infty}{\lim}\mathbb{V}_{n}^{(2)}(L)\right]^{k/2} \\
			                                                                              & = (k-1)!!\left[\sigma_{\chi,\psi}^{2}\right]^{k/2}                                                                                                 \\
			                                                                              & = (k-1)!!\sigma_{\chi,\psi}^{k}.
		\end{split}
	\]
\end{proof}

\section{Proof of Theorem \ref{B bound}} \label{B bound proof}
Let $k$ be a positive integer and let $\underline{r} = (r_{1},...,r_{t}) \vdash k$.
Recall that we defined:
\[
	B(\underline{r}) = \sum_{\substack{
			\underline{\gamma}\in\mathcal{P}_{0}(\underline{r})\\
			\underline{a}\in \mathbb{N}^{k}
		}}H(\underline{\gamma},\underline{a})\mathbb{E}_{n}[U_{n}(\underline{\gamma}^{\underline{a}})].
\]
For $\gamma\in\mathcal{P}_{0}$ and positive integers $a_{1},...,a_{k}$ we denote:
\[
	R(a_{1},...,a_{k})=\underset{n\to\infty}{\lim}\mathbb{E}_{n}\left[U_{n}(\gamma^{a_{1}})...U_{n}(\gamma^{a_{k}})\right].
\]
Corollary \ref{dist cor} shows that the limit above exists and that $R(a_{1},...,a_{k})$ does not depend on the choice of $\gamma\in\mathcal{P}_{0}$.

For:
\[
	\underline{\gamma} = (\overbrace{\gamma_{1},...,\gamma_{1}}^{r_{1}},...,\overbrace{\gamma_{t},...,\gamma_{t}}^{r_{t}}) \in \mathcal{P}_{0}(\underline{r}),
\]
and:
\[
	\underline{a} = (a_{1,1},...,a_{1,r_{1}},...,a_{t,1},...,a_{t,r_{t}}) \in \mathbb{N}^{k},
\]
Corollary \ref{main cor 2} gives:
\begin{multline*}
	\underset{n\to\infty}{\lim}\mathbb{E}_{n}[U_{n}(\underline{\gamma}^{\underline{a}})]=
	\underset{n\to\infty}{\lim}\mathbb{E}_{n}\left[\prod_{i=1}^{t}\prod_{j=1}^{r_{i}}U_{n}(\gamma_{i}^{a_{i,j}})\right] =\\
	\prod_{i=1}^{t}\underset{n\to\infty}{\lim}\mathbb{E}_{n}\left[\prod_{j=1}^{r_{i}}U_{n}(\gamma_{i}^{a_{i,j}})\right] =
	\prod_{i=1}^{t}R(a_{i,1},...,a_{i,r_{i}}).
\end{multline*}
Thus:
\[
	\underset{n\to\infty}{\lim}B(\underline{r}) =
	\sum_{\substack{
			\underline{\gamma}\in\mathcal{P}_{0}(\underline{r})\\
			\underline{a}\in \mathbb{N}^{k}
		}}H(\underline{\gamma},\underline{a})\prod_{i=1}^{t}R(a_{i,1},...,a_{i,r_{i}}),
\]
where the exchange of the sum and the limit is justified as the sum is finite.

For any positive integer $a$ we have:
\[
	R(a) = \underset{n\to\infty}{\lim}\mathbb{E}_{n}\left[U_{n}(\gamma^{a})\right] =
	\underset{n\to\infty}{\lim}\mathbb{E}_{n}\left[F_{n}(\gamma^{a}) - \mathbb{E}_{n}\left[F_{n}(\gamma^{a})\right]\right] = 0.
\]
In particular, if there exists an $i$ for which $r_{i}=1$ then:
\[
	\underset{n\to\infty}{\lim}B(\underline{r}) = 0,
\]
proving the first part of Theorem \ref{B bound}.

As for the case where $r_{i}\geq 2$ for all $i$, we use the following bound:
\begin{lem} \label{R bound}
	Let $a_{1},...,a_{r}$ be positive integers. We have:
	\[
		|R(a_{1},...,a_{r})|\ll a_{1}^{2}...a_{r}^{2},
	\]
	where the implied constant depends on $r$.
\end{lem}
We defer the proof of Lemma \ref{R bound} to the end of the current section.

Let us now resume the proof of Theorem \ref{B bound} given Lemma \ref{R bound}.
Recall that we define:
\[
	H(\gamma_{1},...,\gamma_{k},a_{1},...,a_{k})=
	\prod_{i=1}^{k}\frac{\mathfrak{R}(\chi(\gamma_{i}^{a_{i}}))l_{\gamma_{i}}\hat{\psi}(a_{i}l_{\gamma_{i}}/L)\cos(\alpha a_{i}l_{\gamma_{i}})}{\sinh(a_{i}l_{\gamma_{i}}/2)},
\]
in particular:
\[
	\left|H(\gamma_{1},...,\gamma_{k},a_{1},...,a_{k})\right| \ll l_{\gamma_{1}}...l_{\gamma_{k}}e^{-\frac{a_{1}l_{\gamma_{1}}+...+a_{k}l_{\gamma_{k}}}{2}}.
\]
The fact that $\hat{\psi}$ is supported in $[-1,1]$ implies that $H$ vanishes when there is an $i$ for which $a_{i}\geq \frac{L}{\text{Sys}(X)}$ or an $i$ for which $l_{\gamma_{i}}\geq L$, where $\text{Sys}(X) > 0$ is the systole of $X$, that is, the length of the shortest closed geodesic on $X$.
Thus, given Lemma \ref{R bound}, the computations are straightforward:
\begin{multline*}
	\underset{n\to\infty}{\lim}|B(\underline{r})| \leq
	\sum_{\substack{
			\underline{\gamma}\in\mathcal{P}_{0}(\underline{r})\\
			\underline{a}\in \mathbb{N}^{k}
		}}\left|H(\underline{\gamma},\underline{a})\right|\prod_{i=1}^{t}|R(a_{i,1},...,a_{i,r_{i}})|\\
	\ll \sum_{\substack{\forall_{i} l_{\gamma_{i}} \leq L \\ \forall_{i,j} a_{i,j}\leq \frac{L}{\text{Sys}(X)}}}l_{\gamma_{1}}^{r_{1}}...l_{\gamma_{t}}^{r_{t}}e^{-\frac{\left(\sum_{j=1}^{r_{1}}a_{1,j}\right)l_{\gamma_{1}}+...+\left(\sum_{j=1}^{r_{t}}a_{t,j}\right)l_{\gamma_{t}}}{2}}\prod_{i=1}^{t}|R(a_{i,1},...,a_{i,r_{i}})|\\
	=\sum_{\forall_{i} l_{\gamma_{i}} \leq L}l_{\gamma_{1}}^{r_{1}}...l_{\gamma_{t}}^{r_{t}}\sum_{\forall_{i,j} a_{i,j}\leq \frac{L}{\text{Sys}(X)}}e^{-\frac{\left(\sum_{j=1}^{r_{1}}a_{1,j}\right)l_{\gamma_{1}}+...+\left(\sum_{j=1}^{r_{t}}a_{t,j}\right)l{\gamma_{t}}}{2}}\prod_{i=1}^{t}|R(a_{i,1},...,a_{i,r_{i}})|\\
	\ll\sum_{\forall_{i} l_{\gamma_{i}} \leq L}l_{\gamma_{1}}^{r_{1}}...l_{\gamma_{t}}^{r_{t}}\left(\prod_{i=1}^{t}\sum_{a_{1},...,a_{r_{i}}\geq1}e^{-\frac{(a_{1}+...+a_{r_{i}})l_{\gamma_{i}}}{2}}|R(a_{1},...,a_{r_{i}})|\right).
\end{multline*}

For $l_{\gamma}\gg1$ and a positive integer $r$, Lemma \ref{R bound} gives:
\begin{multline*}
	\sum_{a_{1},...,a_{r}\geq1}e^{-\frac{(a_{1}+...+a_{r})l_{\gamma}}{2}}|R(a_{1},...,a_{r})|\ll\\
	\sum_{a_{1},...,a_{r}\geq1}a_{1}^{2}...a_{r}^{2}e^{-\frac{(a_{1}+...+a_{r})l_{\gamma}}{2}}=
	\left(\sum_{a\geq1}a^{2}e^{-\frac{a l_{\gamma}}{2}}\right)^{r}\ll e^{-\frac{r l_{\gamma}}{2}}.
\end{multline*}
Using this estimate in our bound for $\underset{n\to\infty}{\lim}|B(\underline{r})|$ we get:
\begin{multline*}
	\underset{n\to\infty}{\lim}|B(\underline{r})|
	\ll\sum_{\forall_{i}l_{\gamma_{i}}\leq L}l_{\gamma_{1}}^{r_{1}}...l_{\gamma_{t}}^{r_{t}}\left(\prod_{i=1}^{t}\sum_{a_{1},...,a_{r_{i}}\geq1}e^{-\frac{(a_{1}+...+a_{r_{i}})l_{\gamma_{i}}}{2}}|R(a_{1},...,a_{r_{i}})|\right)\\
	\ll\sum_{\forall_{i}l_{\gamma_{i}}\leq L}l_{\gamma_{1}}^{r_{1}}...l_{\gamma_{t}}^{r_{t}}e^{-\frac{r_{1}l_{\gamma_{1}}+...+r_{t}l_{\gamma_{t}}}{2}}
	=\prod_{i=1}^{t}\sum_{\substack{\gamma\in\mathcal{P}_{0}\\l_{\gamma}\leq L}}l_{\gamma}^{r_{i}}e^{-\frac{r_{i}}{2}l_{\gamma}}.
\end{multline*}

Combining Theorem \ref{prime geo} and summation by parts we get:
\begin{multline*}
	\sum_{\substack{\gamma\in\mathcal{P}_{0}\\ l_{\gamma}\leq L}}l_{\gamma}^{r}e^{-\frac{r}{2}l_{\gamma}}\asymp
	-\int_{\text{Sys}(X)}^{L}\frac{e^{x}}{2x}\left(rx^{r-1}e^{-\frac{r}{2}x}-\frac{r}{2}x^{r}e^{-\frac{r}{2}x}\right)dx=\\
	\int_{\text{Sys}(X)}^{L}\left(\frac{r}{2}x^{r-1}-rx^{r-2}\right)e^{(1-\frac{r}{2})x}dx.
\end{multline*}
Note that we use $\frac{e^{x}}{2x}$ instead of $\frac{e^{x}}{x}$ for the density as we sum over $\mathcal{P}_{0}$ which counts only half of the geodesics.
All in all:
\begin{itemize}
	\item If $r=2$ then:
	      \[
		      \sum_{\substack{\gamma\in\mathcal{P}_{0}\\ l_{\gamma}\leq L}}l_{\gamma}^{r}e^{-\frac{r}{2}l_{\gamma}}\asymp \frac{1}{4}L^{2},
	      \]
	\item while if $r>2$ then:
	      \[
		      \underset{L\to\infty}{\lim}\sum_{\substack{\gamma\in\mathcal{P}_{0}\\ l_{\gamma}\leq L}}l_{\gamma}^{r}e^{-\frac{r}{2}l_{\gamma}}<\infty.
	      \]
\end{itemize}
These observations give:
\[
	\underset{n\to\infty}{\lim}B(\underline{r}) \ll L^{2\#\{r_{i}=2\}},
\]
with the implied constant depending on $\underline{r}$, proving the second part of Theorem \ref{B bound}.

Finally, we give the proof of Lemma \ref{R bound} below.
\begin{proof} [Proof of lemma ~\ref{R bound}]
	Let $\gamma\in\mathcal{P}_{0}$ and let $a_{1},...,a_{r}$ be positive integers.
	Recall that we define:
	\[
		R(a_{1},...,a_{r})=\underset{n\to\infty}{\lim}\mathbb{E}_{n}\left[U_{n}(\gamma^{a_{1}})...U_{n}(\gamma^{a_{r}})\right],
	\]
	and note that the above definition is independent of $\gamma$ as a consequence of Corollary \ref{dist cor}.
	Recalling that:
	\[
		U_{n}(\gamma) = F_{n}(\gamma) - \mathbb{E}_{n}[F_{n}(\gamma)],
	\]
	we have:
	\begin{equation}\label{R eq}
		R(a_{1},...,a_{r})=\underset{n\to\infty}{\lim}\mathbb{E}_{n}\left[\prod_{i}\left(F_{n}(\gamma^{a_{i}})-\mathbb{E}_{n}\left[F_{n}(\gamma^{a_{i}})\right]\right)\right].
	\end{equation}
	To show that:
	\[
		|R(a_{1},...,a_{r})|\ll a_{1}^{2}...a_{r}^{2},
	\]
	it suffices to show that:
	\begin{equation}\label{R assumption}
		\underset{n\to\infty}{\lim}\mathbb{E}_{n}\left[\prod_{i}F_{n}(\gamma^{a_{i}})\right]\ll a_{1}^{2}...a_{r}^{2}.
	\end{equation}

	This follows from the observation that when we open up the RHS of Equation \ref{R eq},
	we get a sum of expectations
	of monomials in the $F_{n}(\gamma^{a_{i}})$. Equation \ref{R assumption} then gives us the bound $a_{1}^{2}...a_{r}^{2}$ on
	all such terms, while
	the number of such terms is some constant (dependent on $r$)
	which implies $R(a_{1},...,a_{r})\ll a_{1}^{2}...a_{r}^{2}$.

	Corollary \ref{dist cor} implies that as $n\to\infty$:
	\[
		\underset{n\to\infty}{\lim}\mathbb{E}_{n}\left[\prod_{i}F_{n}(\gamma^{a_{i}})\right] = \mathbb{E}\left[\prod_{i}\sum_{d|a_{i}}dZ_{1/d}\right],
	\]
	where the $\{Z_{1/d}\}_{d\geq 1}$ are independent Poisson random variables with parameters $1/d$.
	This gives:
	\[
		\underset{n\to\infty}{\lim}\mathbb{E}_{n}\left[\prod_{i}F_{n}(\gamma^{a_{i}})\right] = \mathbb{E}\left[\prod_{i}\sum_{d|a_{i}}dZ_{1/d}\right]
		= \sum_{d_{i}|a_{i}}d_{1}...d_{r}\mathbb{E}\left[Z_{1/d_{1}}...Z_{1/d_{r}}\right].
	\]
	Repeated application of the Cauchy-Schwarz inequality gives:
	\[
		\left|\mathbb{E}\left[Z_{1/d_{1}}...Z_{1/d_{r}}\right]\right|\leq\prod_{i}\mathbb{E}[Z_{1/d_{i}}^{2^{i}}]^{1/2^{i}}\ll\prod_{i}1^{1/2^{i}}=1.
	\]
	This is due to the fact that
	the $m$'th moment of $Z_{\lambda}$ is a polynomial of degree $m$
	in $\lambda$.
	Note that the Poisson variables in the expectation above are not necessarily independent, as some $d_{i}$ could show up more than once in the product.
	Finally, using the bound $\sigma(n) \ll n^{2}$ (see \cite[Theorem 323]{HW} for the bound $\sigma(n) \ll n\log\log n$) we get:
	\[
		\begin{split}
			\left|\underset{n\to\infty}{\lim}\mathbb{E}_{n}\left[\prod_{i}F_{n}(\gamma^{a_{i}})\right]\right| & \ll
			\sum_{d_{i}|a_{i}}d_{1}...d_{r}\left|\mathbb{E}\left[Z_{1/d_{1}}...Z_{1/d_{r}}\right]\right|                                           \\
			                                                                                                  & \ll\sum_{d_{i}|a_{i}}d_{1}...d_{r} \\
			                                                                                                  & =\prod_{i}\sigma(a_{i})            \\
			                                                                                                  & \ll a_{1}^{2}...a_{r}^{2}.
		\end{split}
	\]
\end{proof}

\newpage

\section {Proof of Theorem ~\ref{almost sure}}\label{almost sure proof}
\subsection{Outline of the Proof}
Using the twisted trace formula in Theorem \ref{trace} we express $\mathbb{V}_{T,L,n}$ as a sum over all $\gamma,\delta \in \mathcal{P}_{0}$ and $a,b \in\mathbb{N}$.
We then split the sum into a diagonal term $\text{Diag}$ and an off-diagonal term $\text{Off}$.
The diagonal term $\text{Diag}$ counts the contribution from pairs $(\gamma,\gamma)\in \mathcal{P}_{0}^{2}$,
while the off-diagonal term $\text{Off}$ counts the contribution from pairs $(\gamma,\delta)\in \mathcal{P}_{0}^{2}$ for which $\gamma \neq \delta$.
Subsequently, we show:
\[
	\mathbb{E}_{n}\left|\mathbb{V}_{T,L,n}-\sigma_{\chi,\psi}^{2}\right|
	\leq\frac{4\pi}{L^{2}}\mathbb{E}_{n}\left|\text{Off}\right|+\mathbb{E}_{n}\left|\frac{4\pi}{L^{2}}\text{Diag}-\sigma_{\chi,\psi}^{2}\right|,
\]
then proceed to bound each of the terms in later sections.

To bound $\mathbb{E}_{n}\left|\text{Off}\right|$ we first use the asymptotic independence of the variables $\{U_{n}(\gamma^{a})\}_{\substack{\gamma\in \mathcal{P}_{0}\\ a\geq 1}}$ to reduce the pairs $(\gamma,\delta)\in \mathcal{P}_{0}^{2}$ we sum over in the term $\text{Off}$ to those pairs for which there exist positive integers $a,b,c,d\leq L$ such that $|al_{\gamma}-bl_{\delta}|\ll 1/T$ and $|cl_{\gamma}-dl_{\delta}|\ll 1/T$. We then show that assuming $L=o(T)$ the only such pairs are entries of integer matrices with determinant 0. Finally, we split the sum over such matrices into three sums and then give an upper bound on all three. The final bound on $\mathbb{E}_{n}\left|\text{Off}\right|$ can be found in Proposition \ref{off bound}.

As for the bound on $\mathbb{E}_{n}\left|\frac{4\pi}{L^{2}}\text{Diag}-\sigma_{\chi,\psi}^{2}\right|$, we use the Cauchy-Schwarz inequality to reduce the problem to that of estimating $\mathbb{E}_{n}\left[\text{Diag}\right]$ and $\mathbb{E}_{n}\left[\text{Diag}^{2}\right]$. Using the asymptotic independence of the variables $\{U_{n}(\gamma^{a})\}_{\substack{\gamma\in \mathcal{P}_{0}\\ a\geq 1}}$ and summation by parts, we give estimates that when used produce the bound found in Proposition \ref{diag bound}.

\begin{proof}[Proof of Theorem ~\ref{almost sure}]
	Using the notation from Subsection \ref{trace sec}, we recall that the value of the statistic $N_{\phi}(L)$ for random $\phi\in \text{Hom}(\Gamma,S_{n})$ is denoted by $N_{n}(L)$. Using the trace formula in Theorem \ref{trace}, we recall that $N_{n}(L)$ has the decomposition:
	\[
		N_{n}(L) =N_{n}^{\text{det}}(L) + N_{n}^{\text{osc}}(L),
	\]
	where:
	\[
		N_{n}^{\text{det}}(L) = n(g-1)\int_{\mathbb{R}}h(r)r\tanh(\pi r)dr,
	\]
	is constant, and:
	\[
		N_{n}^{\text{osc}}(L) = \sum_{\substack{\gamma\in\mathcal{P}_{0}\\k\geq1}}\frac{\mathfrak{R}(\chi(\gamma^{k}))l_{\gamma}\hat{h}(kl_{\gamma})}{\sinh(kl_{\gamma}/2)}F_{n}(\gamma^{k}).
	\]
	In addition, recall from the proof of Theorem \ref{main thm 1} that we denote:
	\[
		T_{n}(L)=N_{n}^{\text{osc}}(L)-\mathbb{E}_{n}[N_{n}^{\text{osc}}(L)],
	\]
	and note that as $N^{\text{det}}_{n}(L)$ is independent of the choice of $\phi$ we have:
	\[
		N_{n}(L) - \mathbb{E}_{n}[N_{n}(L)] = T_{n}(L).
	\]
	After using the simple identity:
	\[
		\hat{h}(\zeta)=\frac{2\cos(\alpha\zeta)}{L}\hat{\psi}(\zeta/L),
	\]
	and recalling that we define:
	\[
		U_{n}(\gamma) = F_{n}(\gamma) - \mathbb{E}_{n}\left[F_{n}(\gamma)\right],
	\]
	we have:
	\[
		T_{n}(L)=\frac{2}{L}\sum_{\substack{\gamma\in\mathcal{P}_{0}\\k\geq1}}\frac{\mathfrak{R}(\chi(\gamma^{k}))l_{\gamma}\hat{\psi}(kl_{\gamma}/L)\cos(\alpha kl_{\gamma})}{\sinh(kl_{\gamma}/2)}U_{n}(\gamma^{k}).
	\]

	Taking the expected value with respect to $\alpha$ of $T_{n}(L)$ we see that:
	\[
		\mathbb{E}_{T}\left[T_{n}(L)\right]=\frac{2}{L}\sum_{\substack{\gamma\in\mathcal{P}_{0}\\k\geq1}}\frac{\mathfrak{R}(\chi(\gamma^{k}))l_{\gamma}\hat{\psi}(kl_{\gamma}/L)}{\sinh(kl_{\gamma}/2)}U_{n}(\gamma^{k})\mathbb{E}_{T}\left[\cos(\alpha kl_{\gamma})\right].
	\]
	One easily sees $\mathbb{E}_{T}\left[\cos(\alpha kl_{\gamma})\right]=2\pi\hat{w}\left(Tkl_{\gamma}\right)$
	so that:
	\[
		\mathbb{E}_{T}\left[T_{n}(L)\right]=\frac{4\pi}{L}\sum_{\substack{\gamma\in\mathcal{P}_{0}\\k\geq1}}\frac{\mathfrak{R}(\chi(\gamma^{k}))l_{\gamma}\hat{\psi}(kl_{\gamma}/L)}{\sinh(kl_{\gamma}/2)}U_{n}(\gamma^{k})\hat{w}\left(Tkl_{\gamma}\right).
	\]
	As $kl_{\gamma}\gg1$ we have, due to $\hat{w}$ being compactly supported,
	that:
	\[
		\mathbb{E}_{T}\left[T_{n}(L)\right]=0,
	\]
	for $T\gg1$. Thus for random $\phi$ and $T\gg1$ we have:
	\[
		\begin{split}
			\mathbb{V}_{T,L,n}=\mathbb{V}_{T}[N_{n}(L) - \mathbb{E}_{n}[N_{n}(L)]]=
			\mathbb{V}_{T}\left[T_{n}(L)\right]=\mathbb{E}_{T}\left[T_{n}(L)^{2}\right].
		\end{split}
	\]

	Set:
	\[
		s(\gamma,k)=\frac{\mathfrak{R}(\chi(\gamma^{k}))l_{\gamma}\hat{\psi}(kl_{\gamma}/L)}{\sinh(kl_{\gamma}/2)},
	\]
	and:
	\[
		f_{n}(\gamma,k)=s(\gamma,k)U_{n}(\gamma^{k}) = \frac{\mathfrak{R}(\chi(\gamma^{k}))l_{\gamma}\hat{\psi}(kl_{\gamma}/L)}{\sinh(kl_{\gamma}/2)}U_{n}(\gamma^{k}).
	\]
	Using the identity $\cos x\cos y=\frac{1}{2}\cos(x+y)+\frac{1}{2}\cos(x-y)$ gives:
	\begin{multline*}
		\mathbb{E}_{T}\left[T_{n}(L)^{2}\right]=\\
		\frac{4\pi}{L^{2}}\sum_{\substack{\gamma,\delta\in\mathcal{P}_{0}\\a,b\geq1}}f_{n}(\gamma,a)f_{n}(\delta,b)\left(\hat{w}\left(T(al_{\gamma}-bl_{\delta})\right)+\hat{w}\left(T(al_{\gamma}+bl_{\delta})\right)\right).
	\end{multline*}
	As $al_{\gamma}+bl_{\delta}\gg1$ we have for $T\gg1$ that:
	\[
		\mathbb{E}_{T}\left[T_{n}(L)^{2}\right]=\frac{4\pi}{L^{2}}\sum_{\substack{\gamma,\delta\in\mathcal{P}_{0}\\a,b\geq1}}f_{n}(\gamma,a)f_{n}(\delta,b)\hat{w}\left(T(al_{\gamma}-bl_{\delta})\right).
	\]
	Rewrite the sum above as $\text{Off}+\text{Diag}$ where:
	\[
		\text{Off}=\sum_{\substack{\gamma\neq\delta\\a,b\geq1}}f_{n}(\gamma,a)f_{n}(\delta,b)\hat{w}\left(T(al_{\gamma}-bl_{\delta})\right),
	\]
	and:
	\[
		\text{Diag}=\sum_{\substack{\gamma\in\mathcal{P}_{0}\\a,b\geq1}}f_{n}(\gamma,a)f_{n}(\gamma,b)\hat{w}\left(T(a-b)l_{\gamma}\right),
	\]
	so that:
	\[
		\mathbb{E}_{T}\left[T_{n}(L)^{2}\right]=\frac{4\pi}{L^{2}}\text{Off} + \frac{4\pi}{L^{2}}\text{Diag}.
	\]
	Note that if $a\neq b$ then $|(a-b)l_{\gamma}|\gg1$ so that for
	$T\gg1$ we have:
	\[
		\text{Diag}=\frac{1}{2\pi}\sum_{\substack{\gamma\in\mathcal{P}_{0}\\a\geq1}}f_{n}(\gamma,a)^{2}.
	\]

	Using Markov's inequality we get:
	\[
		\mathbb{P}_{n}\left(\left|\mathbb{V}_{T,L,n}-\sigma_{\chi,\psi}^{2}\right|\geq\epsilon\right)\leq\frac{\mathbb{E}_{n}\left|\mathbb{V}_{T,L,n}-\sigma_{\chi,\psi}^{2}\right|}{\epsilon},
	\]
	while:
	\[
		\begin{split}
			\mathbb{E}_{n}\left|\mathbb{V}_{T,L,n}-\sigma_{\chi,\psi}^{2}\right| & =\mathbb{E}_{n}\left|\mathbb{E}_{T}\left[T_{n}(L)^{2}\right]-\sigma_{\chi,\psi}^{2}\right|                                                   \\
			                                                                     & \leq\frac{4\pi}{L^{2}}\mathbb{E}_{n}\left|\text{Off}\right|+\mathbb{E}_{n}\left|\frac{4\pi}{L^{2}}\text{Diag}-\sigma_{\chi,\psi}^{2}\right|.
		\end{split}
	\]

	Consider the following two propositions:
	\begin{prop}\label{off bound}
		If $L = o(T)$ then for $T\gg 1$:
		\[
			\mathbb{E}_{n}\left|\textnormal{Off}\right| \ll \sqrt{\frac{L^{3}}{T}} + L + \sqrt{O_{L}(1/n)}.
		\]
	\end{prop}

	\begin{prop}\label{diag bound}
		For $T\gg 1$ we have:
		\[
			\mathbb{E}_{n}\left|\frac{4\pi}{L^{2}}\textnormal{Diag}-\sigma_{\chi,\psi}^{2}\right| \ll \frac{1}{\sqrt{L}} + \sqrt{O_{L}(1/n)}.
		\]
	\end{prop}
	We claim that Theorem \ref{almost sure} follows from the above propositions. To see this, recall that:
	\[
		\mathbb{E}_{n}\left|\mathbb{V}_{T,L,n}-\sigma_{\chi,\psi}^{2}\right|\leq\frac{4\pi}{L^{2}}\mathbb{E}_{n}\left|\text{Off}\right|+\mathbb{E}_{n}\left|\frac{4\pi}{L^{2}}\text{Diag}-\sigma_{\chi,\psi}^{2}\right|.
	\]
	Plugging in the bounds from Proposition \ref{off bound} and Proposition \ref{diag bound} we get:
	\begin{multline*}
		\mathbb{E}_{n}\left|\mathbb{V}_{T,L,n}-\sigma_{\chi,\psi}^{2}\right|\ll \\
		\sqrt{\frac{1}{LT}}+\frac{1}{L}+\frac{1}{L^{2}}\sqrt{O_{L}(1/n)}+ \frac{1}{\sqrt{L}} + \sqrt{O_{L}(1/n)},
	\end{multline*}
	as long as $L=o(T)$. Therefore:
	\[
		\underset{n\to\infty}{\limsup}\,\mathbb{E}_{n}\left|\mathbb{V}_{T,L,n}-\sigma_{\chi,\psi}^{2}\right|\ll \frac{1}{\sqrt{L}},
	\]
	so that:
	\[
		\underset{\begin{smallmatrix}L,T\to\infty\\L=o(T)\end{smallmatrix}}{\lim}\underset{n\to\infty}{\limsup}\,\mathbb{P}_{n}\left(\left|\mathbb{V}_{T,L,n}-\sigma_{\chi,\psi}^{2}\right|\geq\epsilon\right)=0,
	\]
	for any $\frac{1}{\epsilon} = o\left(\sqrt{L}\right)$.
\end{proof}

For the rest of the proof, let $\mathcal{P}_{0}^{\leq L}$ denote the set of elements $\gamma$ of $\mathcal{P}_{0}$ for which $l_{\gamma}\leq L$.

\section{Bounding $\mathbb{E}_{n}|\text{Off}|$}
Using the Cauchy-Schwartz inequality we get:
\[
	\mathbb{E}_{n}\left|\text{Off}\right|\leq \sqrt{\mathbb{E}_{n}\left[\text{Off}^{2}\right]}.
\]
Expanding $\text{Off}^{2}$ we see that:
{\small{\[
	\text{Off}^{2} =\\
	\sum_{\substack{\gamma\neq\delta,\alpha\neq\beta\\ a,b,c,d\geq1}}f_{n}(\gamma,a)f_{n}(\delta,b)f_{n}(\alpha,c)f_{n}(\beta,d)\hat{w}\left(T(al_{\gamma}-bl_{\delta})\right)\hat{w}\left(T(cl_{\alpha}-dl_{\beta})\right).
\]}}

Corollary \ref{main cor 2} implies that if $\{\gamma,\delta\}\neq\{\alpha,\beta\}$ then:
\[
	\mathbb{E}_{n}\left[U_{n}(\gamma^{a})U_{n}(\delta^{b})U_{n}(\alpha^{c})U_{n}(\beta^{d})\right] = O_{\gamma,\delta,\alpha,\beta,a,b,c,d}(1/n).
\]
On the flipside, if $\{\gamma,\delta\}=\{\alpha,\beta\}$, say w.l.o.g. that $\gamma = \alpha$ and $\delta = \beta$, then we have:
\begin{multline*}
	\mathbb{E}_{n}\left[U_{n}(\gamma^{a})U_{n}(\delta^{b})U_{n}(\gamma^{c})U_{n}(\delta^{d})\right]=\\
	\underset{n\to\infty}{\lim}\mathbb{E}_{n}\left[U_{n}(\gamma^{a})U_{n}(\gamma^{c})\right]\underset{n\to\infty}{\lim}\mathbb{E}_{n}\left[U_{n}(\delta^{b})U_{n}(\delta^{d})\right] + O_{\gamma,\delta,a,b,c,d}(1/n).
\end{multline*}
Using Lemma \ref{G lem} yields:
\[
	\mathbb{E}_{n}\left[U_{n}(\gamma^{a})U_{n}(\delta^{b})U_{n}(\gamma^{c})U_{n}(\delta^{d})\right] =
	G(a,c)G(b,d) + O_{\gamma,\delta,a,b,c,d}(1/n).
\]

Due to the compact support of $\hat{\psi}$, in $\text{Off}^{2}$ we are summing over all \newline $\gamma,\delta,\alpha,\beta\in \mathcal{P}_{0}^{\leq L}$ and $1\leq a,b,c,d\leq L$.
Thus, if $\{\gamma,\delta\}\neq\{\alpha,\beta\}$ then the implied constants in:
\[
	\mathbb{E}_{n}\left[U_{n}(\gamma^{a})U_{n}(\delta^{b})U_{n}(\alpha^{c})U_{n}(\beta^{d})\right] = O_{\gamma,\delta,\alpha,\beta,a,b,c,d}(1/n),
\]
are uniformly bounded as a function of $L$.
As such:
\[
	\begin{split}
		\mathbb{E}_{n}\left[\text{Off}^{2}\right] & = \sum_{\substack{\gamma\neq\delta,\alpha\neq\beta : \{\gamma,\delta\}=\{\alpha,\beta\} \\a,b,c,d\geq 1}} + \sum_{\substack{\gamma\neq\delta,\alpha\neq\beta : \{\gamma,\delta\}\neq\{\alpha,\beta\}\\a,b,c,d\geq 1}}\\
		                                          & = \sum_{\substack{\gamma\neq\delta,\alpha\neq\beta : \{\gamma,\delta\}=\{\alpha,\beta\} \\a,b,c,d\geq 1}} + O_{L}(1/n)\\
		                                          & \ll \sum_{\substack{\gamma =\alpha\neq\delta =\beta                                     \\ a,b,c,d\geq 1}} + O_{L}(1/n).
	\end{split}
\]

Similar reasoning shows that the implied constants in:
\[
	\mathbb{E}_{n}\left[U_{n}(\gamma^{a})U_{n}(\delta^{b})U_{n}(\gamma^{c})U_{n}(\delta^{d})\right] =
	G(a,c)G(b,d) + O_{\gamma,\delta,a,b,c,d}(1/n),
\]
are uniformly bounded by a function of $L$.
Letting:
{\footnotesize{\[
	M =
	\sum_{\substack{\gamma\neq\delta\\a,b,c,d\geq 1}}G(a,c)G(b,d)s(\gamma,a)s(\gamma,c)s(\delta,b)s(\delta,d)\hat{w}\left(T(al_{\gamma}-bl_{\delta})\right)\hat{w}\left(T(cl_{\gamma}-dl_{\delta})\right),
\]}}
we see that:
\[
	\sum_{\substack{\gamma =\alpha\neq\delta =\beta \\ a,b,c,d\geq 1}}  \ll M + O_{L}(1/n).
\]
All in all:
\[
	\mathbb{E}_{n}\left[\text{Off}^{2}\right] \ll
	M + O_{L}(1/n).
\]

Let us now analyze the sum $M$. Firstly, notice that if:
\[
	\hat{w}\left(T(al_{\gamma}-bl_{\delta})\right)\hat{w}\left(T(cl_{\gamma}-dl_{\delta})\right) \neq 0,
\]
then, due to the compact support of $\hat{w}$, we must have:
\[
	\frac{l_\gamma}{l_\delta} = \frac{b}{a} + \frac{1}{al_{\delta}}O\left(\frac{1}{T}\right) = \frac{d}{c}+ \frac{1}{cl_{\delta}}O\left(\frac{1}{T}\right).
\]
This implies:
\[
	\frac{ad-bc}{ac}=\frac{1}{al_{\delta}}O\left(\frac{1}{T}\right)-\frac{1}{cl_{\delta}}O\left(\frac{1}{T}\right),
\]
or:
\[
	|ad-bc|=\left|\frac{c}{l_{\delta}}O\left(\frac{1}{T}\right)-\frac{a}{l_{\delta}}O\left(\frac{1}{T}\right)\right|.
\]
Recall that we are summing over $1\leq a,b,c,d\leq L$ so that:
\[
	|ad-bc|=\left|\frac{c}{l_{\delta}}O\left(\frac{1}{T}\right)-\frac{a}{l_{\delta}}O\left(\frac{1}{T}\right)\right|\ll \frac{L}{T}.
\]
Choosing $L = o(T)$ implies that for $T\gg 1$, the only options for $a,b,c,d$ are the ones for which $ad-bc = 0$.
Set:
\[
	\text{Sin}_{L} = \left\{\left(\begin{matrix} a&b\\ c&d \end{matrix}\right) \in M_{2}(\mathbb{Z}): 1\leq a,b,c,d\leq L\ ,ad-bc = 0 \right\}.
\]
We have shown that for $T\gg 1$ and $L=o(T)$:
{\footnotesize{
\[
	M =
	\sum_{\substack{\gamma\neq\delta \\ \left(\begin{smallmatrix} a&b\\c&d \end{smallmatrix}\right)\in \text{Sin}_{L}}} G(a,c)G(b,d)s(\gamma,a)s(\gamma,c)s(\delta,b)s(\delta,d)\hat{w}\left(T(al_{\gamma}-bl_{\delta})\right)\hat{w}\left(T(cl_{\gamma}-dl_{\delta})\right).
\]}}

Using the bound:
\[
	|s(\gamma,k)| = \left|\frac{\mathfrak{R}(\chi(\gamma^{k}))l_{\gamma}\hat{\psi}(kl_{\gamma}/L)}{\sinh(kl_{\gamma}/2)}\right| \ll l_{\gamma}e^{-\frac{kl_{\gamma}}{2}},
\]
yields:
{\footnotesize{
	\[
		\begin{split}
			|M| & \ll \sum_{\substack{\gamma\neq\delta \\ \left(\begin{smallmatrix} a&b\\c&d \end{smallmatrix}\right)\in \text{Sin}_{L}}} G(a,c)G(b,d) l_{\gamma}^{2} l_{\delta}^{2}e^{-\frac{(a+c)l_{\gamma}}{2}-\frac{(b+d)l_{\delta}}{2}}\hat{w}\left(T(al_{\gamma}-bl_{\delta})\right)\hat{w}\left(T(cl_{\gamma}-dl_{\delta})\right)\\
			    & =\sum_{\substack{\gamma\neq\delta    \\ \left(\begin{smallmatrix} a&b\\c&d \end{smallmatrix}\right) = \left(\begin{smallmatrix} 1&1\\1&1 \end{smallmatrix}\right)}} +
			\sum_{\substack{\gamma\neq\delta           \\ \left(\begin{smallmatrix} 1&1\\c&d \end{smallmatrix}\right) \in \text{Sin}_{L} \\ c+d\geq 3}} +
			\sum_{\substack{\gamma\neq\delta           \\ \left(\begin{smallmatrix} a&b\\1&1 \end{smallmatrix}\right) \in \text{Sin}_{L} \\ a+b\geq 3}}+
			\sum_{\substack{\gamma\neq\delta           \\ \left(\begin{smallmatrix} a&b\\c&d \end{smallmatrix}\right) \in \text{Sin}_{L} \\ a+b,c+d\geq 3}}\\
			    & \ll \sum_{\substack{\gamma\neq\delta \\ \left(\begin{smallmatrix} a&b\\c&d \end{smallmatrix}\right) = \left(\begin{smallmatrix} 1&1\\1&1 \end{smallmatrix}\right)}} +
			\sum_{\substack{\gamma\neq\delta           \\ \left(\begin{smallmatrix} 1&1\\c&d \end{smallmatrix}\right) \in \text{Sin}_{L} \\ c+d\geq 3}} +
			\sum_{\substack{\gamma\neq\delta           \\ \left(\begin{smallmatrix} a&b\\c&d \end{smallmatrix}\right) \in \text{Sin}_{L} \\ a+b,c+d\geq 3}}.
		\end{split}
	\]}}
Denote the last three sums as $\Sigma_{1}, \Sigma_{2}$, and $\Sigma_{3}$ respectively.
That is:
\[
	\Sigma_{1}\stackrel{def}{=}\sum_{\substack{\gamma\neq\delta \\ \left(\begin{smallmatrix} a&b\\c&d \end{smallmatrix}\right) = \left(\begin{smallmatrix} 1&1\\1&1 \end{smallmatrix}\right)}},
	\Sigma_{2}\stackrel{def}{=}\sum_{\substack{\gamma\neq\delta  \\ \left(\begin{smallmatrix} 1&1\\c&d \end{smallmatrix}\right) \in \text{Sin}_{L} \\ c+d\geq 3}},
	\Sigma_{3}\stackrel{def}{=}\sum_{\substack{\gamma\neq\delta \\ \left(\begin{smallmatrix} a&b\\c&d \end{smallmatrix}\right) \in \text{Sin}_{L} \\ a+b,c+d\geq 3}},
\]
where for all three sums $\gamma$ and $\delta$ are elements of $\mathcal{P}_{0}^{\leq L}$.
Thus for $T\gg 1$ and $L=o(T)$ we have:
\[
	\mathbb{E}_{n}\left[\text{Off}^{2}\right] \ll \Sigma_{1} + \Sigma_{2} + \Sigma_{3} + O_{L}(1/n).
\]

To finish bounding $\mathbb{E}_{n}\left[\text{Off}^{2}\right]$ we bound each $\Sigma_{i}$ individually.
\begin{lem}\label{sigma1 bound}
	If $L=o(T)$ then for $T\gg 1$ we have:
	\[
		\Sigma_{1} \ll \frac{L^{3}}{T}+1.
	\]
\end{lem}
\begin{lem}\label{sigma2 bound}
	If $L=o(T)$ then for $T\gg 1$ we have:
	\[
		\Sigma_{2} \ll 1.
	\]
\end{lem}
\begin{lem}\label{sigma3 bound}
	If $L=o(T)$ then for $T\gg 1$ we have:
	\[
		\Sigma_{3} \ll L^{2}.
	\]
\end{lem}

The above lemmas give:
\[
	\mathbb{E}_{n}\left[\text{Off}^{2}\right] \ll \Sigma_{1} + \Sigma_{2} + \Sigma_{3}+ O_{L}(1/n) \ll
	\frac{L^{3}}{T} + L^{2} + O_{L}(1/n),
\]
so that:
\[
	\mathbb{E}_{n}\left|\text{Off}\right| \ll \sqrt{\frac{L^{3}}{T}} + L + \sqrt{O_{L}(1/n)},
\]
which is exactly Proposition \ref{off bound}.

\subsection{Bounding $\Sigma_{1}$}
Let us tackle $\Sigma_{1}$, using $G(1,1) = 1$ we have:
\[
	\Sigma_{1}=\sum_{\substack{\gamma\neq\delta \\ \left(\begin{smallmatrix} a&b\\c&d \end{smallmatrix}\right) = \left(\begin{smallmatrix} 1&1\\1&1 \end{smallmatrix}\right)}}
	=\sum_{\gamma\neq\delta} l_{\gamma}^{2} l_{\delta}^{2}e^{-l_{\gamma}-l_{\delta}}\hat{w}\left(T(l_{\gamma}-l_{\delta})\right)^{2},
\]
where we sum over $\gamma$ and $\delta$ that are elements of $\mathcal{P}_{0}^{\leq L}$.

A simple bound is:
\[
	\Sigma_{1}
	\ll\sum_{\substack{\gamma\neq\delta \\ \left|l_{\gamma}-l_{\delta}\right|\ll 1/T}} l_{\gamma}^{2} l_{\delta}^{2}e^{-l_{\gamma}-l_{\delta}}.
\]
If $\left|l_{\gamma}-l_{\delta}\right|\ll 1/T$ then $l_{\delta} = l_{\gamma} + O\left(\frac{1}{T}\right)$, thus:
\[
	\sum_{\substack{\gamma\neq\delta \\ \left|l_{\gamma}-l_{\delta}\right|\ll 1/T}} l_{\gamma}^{2} l_{\delta}^{2}e^{-l_{\gamma}-l_{\delta}}\ll
	\sum_{\gamma\in\mathcal{P}_{0}^{\leq L}} l_{\gamma}^{4}e^{-2l_{\gamma}}\#\left\{\delta :l_{\delta} = l_{\gamma} + O\left(\frac{1}{T}\right)\right\}.
\]
Using Theorem \ref{prime geo} we have the bound:
\[
	\#\left\{\delta :l_{\delta} = l_{\gamma} + O\left(\frac{1}{T}\right)\right\} \ll \frac{1}{T}\frac{e^{l_{\gamma}}}{l_{\gamma}} +\frac{e^{\nu l_{\gamma}}}{l_{\gamma}},
\]
for $3/4\leq \nu < 1$ dependent on $X$.  Plugging this back in we get:
\begin{multline*}
	\sum_{\gamma\in\mathcal{P}_{0}^{\leq L}} l_{\gamma}^{4}e^{-2l_{\gamma}}\#\left\{\delta :l_{\delta} = l_{\gamma} + O\left(\frac{1}{T}\right)\right\} \ll\\
	\sum_{\gamma\in\mathcal{P}_{0}^{\leq L}} l_{\gamma}^{4}e^{-2l_{\gamma}}\left(\frac{1}{T}\frac{e^{l_{\gamma}}}{l_{\gamma}} + \frac{e^{\nu l_{\gamma}}}{l_{\gamma}}\right)=
	\frac{1}{T}\sum_{\gamma\in\mathcal{P}_{0}^{\leq L}}  l_{\gamma}^{3}e^{-l_{\gamma}} + \sum_{\gamma\in\mathcal{P}_{0}^{\leq L}}  l_{\gamma}^{3}e^{-(2-\nu)l_{\gamma}}.
\end{multline*}

Summation by parts gives:
\[
	\sum_{\gamma\in\mathcal{P}_{0}^{\leq L}}  l_{\gamma}^{3}e^{-l_{\gamma}} \ll \int_{0}^{L}x^{2} \ll L^{3},
\]
while $\nu<1$ implies:
\[
	\sum_{\gamma\in\mathcal{P}_{0}^{\leq L}}  l_{\gamma}^{3}e^{-(2-\nu)l_{\gamma}} \ll \int_{0}^{L}x^{2}e^{-(1-\nu)x} \ll 1.
\]

All in all:
\[
	\Sigma_{1} \ll \frac{L^{3}}{T} + 1,
\]
which is Lemma \ref{sigma1 bound}.

\subsection{Bounding $\Sigma_{2}$}
Recall that:
\begin{multline*}
	\Sigma_{2} =\\
	\sum_{\gamma\neq\delta}\sum_{\substack{\left(\begin{smallmatrix} 1&1\\c&d \end{smallmatrix}\right) \in \text{Sin}_{L} \\ c+d\geq 3}} G(1,c)G(1,d) l_{\gamma}^{2} l_{\delta}^{2}e^{-\frac{(1+c)l_{\gamma}}{2}-\frac{(1+d)l_{\delta}}{2}}\hat{w}\left(T(l_{\gamma}-l_{\delta})\right)\hat{w}\left(T(cl_{\gamma}-dl_{\delta})\right),
\end{multline*}
where we sum over $\gamma$ and $\delta$ that are elements of $\mathcal{P}_{0}^{\leq L}$.

If $\left(\begin{smallmatrix} 1&1\\c&d \end{smallmatrix}\right) \in \text{Sin}_{L}$ then $c=d$, while $c+d\geq 3$ implies $c=d\geq 2$. Also notice that $G(1,t)=1$ for all positive integers $t$, therefore:
\begin{multline*}
	\Sigma_{2}=\sum_{\gamma\neq\delta}\sum_{2\leq c\leq L} l_{\gamma}^{2} l_{\delta}^{2}e^{-\frac{(1+c)l_{\gamma}}{2}-\frac{(1+c)l_{\delta}}{2}}\hat{w}\left(T(l_{\gamma}-l_{\delta})\right)\hat{w}\left(cT(l_{\gamma}-l_{\delta})\right) \ll\\
	\sum_{\gamma\neq\delta}\sum_{c\geq 2} l_{\gamma}^{2} l_{\delta}^{2}e^{-\frac{(1+c)l_{\gamma}}{2}-\frac{(1+c)l_{\delta}}{2}}\hat{w}\left(T(l_{\gamma}-l_{\delta})\right)\hat{w}\left(cT(l_{\gamma}-l_{\delta})\right).
\end{multline*}
In order for $\gamma\neq\delta$ to give a non-zero contribution we need $\left|l_{\gamma}-l_{\delta}\right|\ll \frac{1}{cT}$ (because $c\geq 2$ this  automatically implies $\left|l_{\gamma}-l_{\delta}\right|\ll \frac{1}{T}$). Thus, $\gamma\neq\delta$ give a non-zero contribution if and only if $l_{\delta} = l_{\gamma} + \frac{1}{c}O\left(\frac{1}{T}\right)$. This implies:
\[
	\Sigma_{2} \ll \sum_{\gamma\in\mathcal{P}_{0}^{\leq L}}\sum_{c\geq 2} l_{\gamma}^{4}e^{-(1+c)l_{\gamma}}\#\left\{\delta : l_{\delta} = l_{\gamma} +  \frac{1}{c}O\left(\frac{1}{T}\right)\right\}.
\]

As before, using Theorem \ref{prime geo} gives:
\[
	\#\left\{\delta : l_{\delta} = l_{\gamma} +  \frac{1}{c}O\left(\frac{1}{T}\right)\right\} \ll \frac{1}{T}\frac{e^{l_{\gamma}}}{l_{\gamma}} +\frac{e^{\nu l_{\gamma}}}{l_{\gamma}},
\]
so that:
\begin{multline*}
	\sum_{\gamma\in\mathcal{P}_{0}^{\leq L}}\sum_{c\geq 2} l_{\gamma}^{4}e^{-(1+c)l_{\gamma}}\#\left\{\delta : l_{\delta} = l_{\gamma} +  \frac{1}{c}O\left(\frac{1}{T}\right)\right\} \ll\\
	\frac{1}{T}\sum_{\gamma\in\mathcal{P}_{0}^{\leq L}}\sum_{c\geq 2} l_{\gamma}^{3}e^{-cl_{\gamma}} + \sum_{\gamma\in\mathcal{P}_{0}^{\leq L}}\sum_{c\geq 2} l_{\gamma}^{3}e^{-((1-\nu)+c)l_{\gamma}}.
\end{multline*}

For $l_{\gamma}\gg 1$ we have:
\[
	\sum_{c\geq2} e^{-cl_{\gamma}}\ll e^{-2l_{\gamma}},
\]
and:
\[
	\sum_{c\geq2}e^{-((1-\nu)+c)l_{\gamma}}\ll e^{-(3-\nu)l_{\gamma}}.
\]
Using summation by parts we get:
\[
	\sum_{\gamma\in\mathcal{P}_{0}^{\leq L}}\sum_{c\geq 2} l_{\gamma}^{3}e^{-cl_{\gamma}}\ll
	\sum_{\gamma\in\mathcal{P}_{0}^{\leq L}}l_{\gamma}^{3}e^{-2l_{\gamma}}\ll 1,
\]
while $\nu<1$ implies:
\[
	\sum_{\gamma\in\mathcal{P}_{0}^{\leq L}}\sum_{c\geq 2} l_{\gamma}^{3}e^{-((1-\nu)+c)l_{\gamma}} \ll
	\sum_{\gamma\in\mathcal{P}_{0}^{\leq L}} l_{\gamma}^{3}e^{-(3-\nu)l_{\gamma}}  \ll 1.
\]

Overall:
\[
	\Sigma_{2} \ll 1,
\]
and thus we have Lemma \ref{sigma2 bound}.

\section{Bounding $\Sigma_{3}$}
Recall that $\Sigma_{3}$ is the following sum:
\[\scalebox{.85}{$\displaystyle
		\sum_{\gamma\neq\delta}\sum_{\substack{\left(\begin{smallmatrix} a&b\\c&d \end{smallmatrix}\right) \in \text{Sin}_{L} \\ a+b,c+d\geq 3}} G(a,c)G(b,d)s(\gamma,a)s(\gamma,c)s(\delta,b)s(\delta,d)\hat{w}\left(T(al_{\gamma}-bl_{\delta})\right)\hat{w}\left(T(cl_{\gamma}-dl_{\delta})\right),$}
\]
where $\gamma$ and $\delta$ range over $\mathcal{P}_{0}^{\leq L}$.

If $A=\left(\begin{smallmatrix} a&b\\c&d \end{smallmatrix}\right) \in \text{Sin}_{L}$ then one of $A$'s columns is a multiple of the other, thus:
\[
	\Sigma_{3} = \sum_{\gamma\neq\delta}\sum_{\substack{\left(\begin{smallmatrix} a&b\\c&d \end{smallmatrix}\right) \in \text{Sin}_{L} \\ a+b,c+d\geq 3}} =
	\sum_{\gamma\neq\delta}\sum_{\substack{\exists \lambda\neq 0: \left(\begin{smallmatrix} b\\d\end{smallmatrix}\right)=\lambda \left(\begin{smallmatrix} a\\c\end{smallmatrix}\right) \\ 1\leq a,b,c,d \leq L \\ a+b,c+d\geq 3}}.
\]
We denote the last double sum by $\sum_{\text{Columns}}$ so that:
\[
	\Sigma_{3} \ll \sum_{\text{Columns}}.
\]

Before delving into the analysis, note that the scalar $\lambda$ found in the definition of $\sum_{\text{Columns}}$
is a non-negative rational number of the form $\frac{u}{v}$ where $1\leq u,v\leq L$.
We denote the set of such rational numbers by $\mathbb{Q}_{L}$ and set:
\[
	\mathbb{Q}_{\geq x,L} = \mathbb{Q}_{L}\cap [x,\infty), \;\;\;\;\; \mathbb{Q}_{> x,L} = \mathbb{Q}_{L}\cap (x,\infty).
\]
Also, note that $\#\mathbb{Q}_{L}\ll L^{2}$ as $\#\mathbb{Q}_{L}$ is at most the number of pairs of positive integers not exceeding $L$.
In addition, let us extend the definition of \newline $G(a,b) = \sigma(\gcd(a,b))$ so that $G(x,y) = 0$ if $x$ or $y$ are not integers.

Let us bound $\sum_{\text{Columns}}$. As one column is a multiple of the other we can assume w.l.g. that $ \left(\begin{smallmatrix} b\\d\end{smallmatrix}\right)=\lambda \left(\begin{smallmatrix} a\\c\end{smallmatrix}\right) $ for some $ \lambda \in \mathbb{Q}_{\geq 1,L}$ so that:
\begin{multline*}
	\sum_{\text{Columns}} \ll\\
	\scalebox{.85}{$\displaystyle
		\sum_{\gamma\neq\delta}\sum_{\substack{\lambda \in \mathbb{Q}_{\geq 1,L}\\ 1\leq a,c \leq L \\ (1+\lambda)a, (1+\lambda)c\geq 3}} G(a,c)G(\lambda a,\lambda c)l_{\gamma}^{2}l_{\delta}^{2}e^{-\frac{a+c}{2}l_{\gamma} - \lambda\frac{a+c}{2}l_{\delta}}\hat{w}\left(aT(l_{\gamma}-\lambda l_{\delta})\right)\hat{w}\left(cT(l_{\gamma}-\lambda l_{\delta})\right).$}
\end{multline*}
We can assume that $c\leq a$ in the above sum and so:
\[
	\sum_{\text{Columns}} \ll
	\sum_{\gamma\neq\delta}\sum_{\substack{\lambda \in \mathbb{Q}_{\geq 1,L}\\ 1\leq c\leq a \leq L\\ (1+\lambda)a, (1+\lambda)c\geq 3}}.
\]

If $\gamma,\delta$ are to contribute to the sum we need $l_{\gamma}=\lambda l_{\delta} + \frac{1}{a}O\left(\frac{1}{T}\right)$ so that:
{\footnotesize{\begin{multline*}
		\sum_{\text{Columns}}\ll\\
		\sum_{\delta\in\mathcal{P}_{0}^{\leq L}}\sum_{\substack{\lambda \in \mathbb{Q}_{\geq 1,L}\\ 1\leq c\leq a\leq L \\ (1+\lambda)a, (1+\lambda)c\geq 3}} \lambda^{2}G(a,c)G(\lambda a,\lambda c)l_{\delta}^{4}e^{- \lambda(a+c)l_{\delta}}\#\left\{\gamma: l_{\gamma}=\lambda l_{\delta} + \frac{1}{a}O\left(\frac{1}{T}\right)\right\}.
	\end{multline*}}}Using Theorem \ref{prime geo} we have:
\[
	\#\left\{\gamma: l_{\gamma}=\lambda l_{\delta} + \frac{1}{a}O\left(\frac{1}{T}\right)\right\} \ll \frac{1}{aT}\frac{e^{\lambda l_{\delta}}}{\lambda l_{\delta}} + \frac{e^{\nu \lambda l_{\delta}}}{\lambda l_{\delta}},
\]
thus:
\begin{multline*}
	\sum_{\text{Columns}} \ll \frac{1}{T} \sum\frac{\lambda}{a}G(a,c)G(\lambda a,\lambda c)l_{\delta}^{3}e^{- \lambda(a+c-1)l_{\delta}} +\\
	\sum\lambda G(a,c)G(\lambda a,\lambda c)l_{\delta}^{3}e^{- \lambda(a+c-\nu)l_{\delta}},
\end{multline*}
where the sums in the RHS are over all $\delta \in \mathcal{P}_{0}^{\leq L}$, $\lambda \in \mathbb{Q}_{\geq 1,L}$ and $1\leq c\leq a\leq L$ such that:
\[
	(1+\lambda)a, (1+\lambda)c\geq 3.
\]

We denote:
\[
	J_{1} = \sum_{\delta\in\mathcal{P}_{0}^{\leq L}}\sum_{\substack{\lambda \in \mathbb{Q}_{\geq 1,L}\\ 1\leq c\leq a\leq L\\ (1+\lambda)a\geq 3, (1+\lambda)c\geq 3}}\frac{\lambda}{a}G(a,c)G(\lambda a,\lambda c)l_{\delta}^{3}e^{- \lambda(a+c-1)l_{\delta}},
\]
and:
\[
	J_{2} = \sum_{\delta\in\mathcal{P}_{0}^{\leq L}}\sum_{\substack{\lambda \in \mathbb{Q}_{\geq 1,L}\\ 1\leq c\leq a\leq L\\ (1+\lambda)a\geq 3, (1+\lambda)c\geq 3}}\lambda G(a,c)G(\lambda a,\lambda c)l_{\delta}^{3}e^{- \lambda(a+c-\nu)l_{\delta}}.
\]
Overall we are left with:
\[
	\sum_{\text{Columns}} \ll  \frac{1}{T}J_{1} + J_{2}.
\]
\subsection{Bounding $J_{1}$}
Recall that the inner sum in $J_{1}$ is over $\lambda \in \mathbb{Q}_{\geq 1,L}$ and $1\leq c\leq a\leq L$ such that:
\[
	(1+\lambda)a, (1+\lambda)c\geq 3.
\]
If $\lambda < 2$ then for the conditions in the sum to be satisfied we need $2\leq c\leq a\leq L$. Write $J_{1} = J_{1,\lambda < 2} + J_{1,\lambda \geq 2}$, where $J_{1,\lambda < 2}, J_{1,\lambda \geq 2}$ are the sum $J_{1}$ with the extra condition on $\lambda$ respectively.

First we deal with $J_{1,\lambda < 2}$, writing it explicitly:
\[
	J_{1,\lambda < 2} = \sum_{\delta\in\mathcal{P}_{0}^{\leq L}}\sum_{\substack{\lambda \in \mathbb{Q}_{L}\cap[1,2) \\ 2\leq c\leq a\leq L}}\frac{\lambda}{a}G(a,c)G(\lambda a,\lambda c)l_{\delta}^{3}e^{- \lambda(a+c-1)l_{\delta}}.
\]
For $\lambda \in \mathbb{Q}_{L}\cap[1,2)$ and $2\leq c\leq a$ as in the inner sum of $J_{1,\lambda < 2}$ we have:
\[
	G(\lambda a,\lambda c)\ll \lambda a \log(\lambda a)\ll a\log(a),
\]
so that for $l_{\delta}\gg 1$:
\[
	\sum_{\substack{\lambda \in \mathbb{Q}_{L}\cap[1,2)  \\ 2\leq c\leq a\leq L}}\frac{\lambda}{a}G(a,c)G(\lambda a,\lambda c)e^{- \lambda(a+c-1)l_{\delta}} \ll L^{2}e^{-3l_{\delta}}.
\]
Thus:
\[
	J_{1,\lambda < 2} \ll L^{2}\sum_{\delta\in\mathcal{P}_{0}^{\leq L}}l_{\delta}^{3}e^{- 3l_{\delta}}  \ll L^{2}.
\]

As for $J_{1,\lambda \geq 2}$, if $\lambda \geq 2$ then for the conditions to hold we can have any $1\leq c\leq a\leq L$ resulting in:
\[
	J_{1,\lambda \geq 2} = \sum_{\delta\in\mathcal{P}_{0}^{\leq L}}\sum_{\substack{\lambda \in \mathbb{Q}_{\geq 2,L} \\ 1\leq c\leq a\leq L}}\frac{\lambda}{a}G(a,c)G(\lambda a,\lambda c)l_{\delta}^{3}e^{- \lambda(a+c-1)l_{\delta}}.
\]
For $l_{\delta}\gg 1$ trivially:
\[
	\sum_{\substack{\lambda \in \mathbb{Q}_{\geq 2,L} \\ 1\leq c\leq a\leq L}} \frac{\lambda}{a}G(a,c)G(\lambda a,\lambda c)e^{- \lambda(a+c-1)l_{\delta}} \ll L^{2}e^{-2l_{\delta}},
\]
so that:
\begin{multline*}
	J_{1,\lambda \geq 2} = \sum_{\delta\in\mathcal{P}_{0}^{\leq L}}\sum_{\substack{\lambda \in \mathbb{Q}_{\geq 2,L} \\ 1\leq c\leq a\leq L}}\frac{\lambda}{a}G(a,c)G(\lambda a,\lambda c)l_{\delta}^{3}e^{- \lambda(a+c-1)l_{\delta}} \ll\\
	L^{2}\sum_{\delta\in\mathcal{P}_{0}^{\leq L}}l_{\delta}^{3}e^{-2l_{\delta}} \ll L^{2}.
\end{multline*}
All in all:
\[
	J_{1} = J_{1,\lambda < 2} + J_{1,\lambda \geq 2} \ll L^{2}.
\]
\subsection{Bounding $J_{2}$}
Analogously to the analysis done for $J_{1}$, we write \newline $J_{2} = J_{2,\lambda < 2} + J_{2,\lambda\geq 2}$ where:
\[
	J_{2,\lambda < 2} = \sum_{\delta\in\mathcal{P}_{0}^{\leq L}}\sum_{\substack{\lambda \in \mathbb{Q}_{L}\cap[1,2) \\ 2\leq c\leq a\leq L}}\lambda G(a,c)G(\lambda a,\lambda c)l_{\delta}^{3}e^{- \lambda(a+c-\nu)l_{\delta}},
\]
and:
\[
	J_{2,\lambda\geq 2} = \sum_{\delta\in\mathcal{P}_{0}^{\leq L}}\sum_{\substack{\lambda \in \mathbb{Q}_{\geq 2,L} \\ 1\leq c\leq a\leq L}} \lambda G(a,c)G(\lambda a,\lambda c)l_{\delta}^{3}e^{- \lambda(a+c-\nu)l_{\delta}}.
\]

Let us first deal with $J_{2,\lambda < 2}$. Fixing $l_{\delta}\gg 1$ we have as before:
\[
	\sum_{\substack{\lambda \in \mathbb{Q}_{L}\cap[1,2) \\ 2\leq c\leq a\leq L}}\lambda G(a,c)G(\lambda a,\lambda c)e^{- \lambda(a+c-\nu)l_{\delta}} \ll L^{2}e^{-(4-\nu)l_{\delta}},
\]
thus giving:
\[
	J_{2,\lambda<2} \ll L^{2}\sum_{\delta\in\mathcal{P}_{0}^{\leq L}} l_{\delta}^{3}e^{-(4-\nu)l_{\delta}} \ll L^{2}.
\]

As for $J_{2,\lambda\geq 2}$, fixing $l_{\delta}\gg 1$ we have:
\[
	\sum_{\substack{\lambda  \in \mathbb{Q}_{\geq 2,L} \\ 1\leq c\leq a\leq L}} \lambda G(a,c)G(\lambda a,\lambda c)e^{- \lambda(a+c-\nu)l_{\delta}} \ll L^{2}e^{-2(2-\nu)l_{\delta}},
\]
so that:
\[
	J_{2,\lambda\geq 2} \ll L^{2}\sum_{\delta\in\mathcal{P}_{0}^{\leq L}} l_{\delta}^{3}e^{-2(2-\nu)l_{\delta}} \ll L^{2}.
\]
All in all:
\[
	J_{2} = J_{2,\lambda<2} + J_{2,\lambda\geq 2} \ll L^{2}.
\]

Finally:
\[
	\Sigma_{3} \ll \sum_{\text{Columns}} \ll  \frac{1}{T}J_{1} + J_{2} \ll L^{2},
\]
proving Lemma \ref{sigma3 bound}.

\section{Bounding $\mathbb{E}_{n}\left|\frac{4\pi}{L^{2}}\text{Diag}-\sigma_{\chi,\psi}^{2}\right|$}
\begin{proof}[Proof of Proposition \ref{diag bound}]
	To bound $\mathbb{E}_{n}\left|\frac{4\pi}{L^{2}}\text{Diag}-\sigma_{\chi,\psi}^{2}\right|$
	we use Cauchy-Schwarz:
	\[
		\left[\mathbb{E}_{n}\left|\text{Diag}-\sigma_{\chi,\psi}^{2}\right|\right]^{2}\leq
		\mathbb{E}_{n}\left[\left(\frac{4\pi}{L^{2}}\text{Diag}-\sigma_{\chi,\psi}^{2}\right)^{2}\right].
	\]
	Squaring out we get:
	\[
		\mathbb{E}_{n}\left[\left(\frac{4\pi}{L^{2}}\text{Diag}-\sigma_{\chi,\psi}^{2}\right)^{2}\right]=\frac{16\pi^{2}}{L^{4}}\mathbb{E}_{n}\left[\text{Diag}^{2}\right]-\frac{8\pi\sigma_{\chi,\psi}^{2}}{L^{2}}\mathbb{E}_{n}\left[\text{Diag}\right]+\sigma_{\chi,\psi}^{4}.
	\]

	We show the following two lemmas:
	\begin{lem}\label{exp diag}
		For $T\gg 1$ we have:
		\[
			\mathbb{E}_{n}[\textnormal{Diag}]=\frac{L^{2}}{4\pi}\sigma_{\chi,\psi}^{2}+O(L) + O_{L}(1/n).
		\]
	\end{lem}
	\begin{lem}\label{exp diag2}
		For $T\gg 1$ we have:
		\[
			\mathbb{E}_{n}[\textnormal{Diag}^{2}]=\frac{L^{4}}{16\pi^{2}}\sigma_{\chi,\psi}^{4}+O(L^{3})+O_{L}(1/n).
		\]
	\end{lem}

	Together, the above lemmas imply:
	\[
		\mathbb{E}_{n}\left[\left(\frac{4\pi}{L^{2}}\text{Diag}-\sigma_{\chi,\psi}^{2}\right)^{2}\right] = O(1/L)+O_{L}(1/n),
	\]
	so that:
	\[
		\mathbb{E}_{n}\left|\frac{4\pi}{L^{2}}\text{Diag}-\sigma_{\chi,\psi}^{2}\right| \ll \frac{1}{\sqrt{L}} + \sqrt{O_{L}(1/n)},
	\]
	which is exactly Proposition \ref{diag bound}.
\end{proof}

\subsection{Estimating $\mathbb{E}_{n}\left[\text{Diag}\right]$}
\begin{proof} [Proof of Lemma \ref{exp diag}]
	Recall that for $T\gg 1$ we have:
	\[
		\text{Diag} = \frac{1}{2\pi}\sum_{\substack{\gamma\in\mathcal{P}_{0}\\a\geq1}}f_{n}(\gamma,a)^{2},
	\]
	where:
	\[
		f_{n}(\gamma,a)=\frac{\mathfrak{R}(\chi(\gamma^{a}))l_{\gamma}\hat{\psi}(al_{\gamma}/L)}{\sinh(al_{\gamma}/2)}U_{n}(\gamma^{a}) = s(\gamma,a)U_{n}(\gamma^{a}).
	\]
	Thus:
	\[
		\mathbb{E}_{n}\left[\text{Diag}\right]=\frac{1}{2\pi}\left[\sum_{\gamma\in\mathcal{P}_{0}}\sum_{a=1}\mathbb{E}_{n}\left[f_{n}(\gamma,a)^{2}\right]+\sum_{\gamma\in\mathcal{P}_{0}}\sum_{a\geq2}\mathbb{E}_{n}\left[f_{n}(\gamma,a)^{2}\right]\right].
	\]
	Recalling that $\hat{\psi}$ is supported in $[-1,1]$ we have:
	\[
		\mathbb{E}_{n}\left[\text{Diag}\right]=\frac{1}{2\pi}\left[\sum_{\gamma\in\mathcal{P}_{0}^{\leq L}}\sum_{a=1}\mathbb{E}_{n}\left[f_{n}(\gamma,a)^{2}\right]+\sum_{\gamma\in\mathcal{P}_{0}^{\leq L}}\sum_{2\leq a\leq L}\mathbb{E}_{n}\left[f_{n}(\gamma,a)^{2}\right]\right].
	\]

	To estimate the expression $\mathbb{E}_{n}\left[f_{n}(\gamma,a)^{2}\right]$ we combine Corollary \ref{main cor 2} with Corollary \ref{dist cor}, together they yield:
	\begin{multline*}
		\mathbb{E}_{n}\left[f_{n}(\gamma,a)^{2}\right]=\\
		\left(\frac{\mathfrak{R}(\chi(\gamma^{a}))l_{\gamma}\hat{\psi}(al_{\gamma}/L)}{\sinh(al_{\gamma}/2)}\right)^{2}\mathbb{E}\left[\left(\sum_{d|a}(dZ_{1/d}-1)\right)^{2}\right] + O_{\gamma,a}(1/n),
	\end{multline*}
	where $\{Z_{1/d}\}_{d\geq 1}$ are independent Poisson random variables with parameters $1/d$.
	As $dZ_{1/d}-1$ are independent for different $d$ this comes out
	to:
	\begin{lem}\label{f exp} We have:
		\begin{multline*}
			\mathbb{E}_{n}\left[f_{n}(\gamma,a)^{2}\right]= s(\gamma,a)^{2}\sigma(a) + O_{\gamma, a}(1/n) =\\
			\left(\frac{\mathfrak{R}(\chi(\gamma^{a}))l_{\gamma}\hat{\psi}(al_{\gamma}/L)}{\sinh(al_{\gamma}/2)}\right)^{2}\left(\sum_{d|a}d\right)+ O_{\gamma,a}(1/n).
		\end{multline*}
	\end{lem}

	Denote:
	\[
		\Theta_{1} = \sum_{\gamma\in\mathcal{P}_{0}^{\leq L}}\sum_{a=1}\mathbb{E}_{n}\left[f_{n}(\gamma,a)^{2}\right] = \sum_{\gamma\in\mathcal{P}_{0}^{\leq L}}\mathbb{E}_{n}\left[f_{n}(\gamma,1)^{2}\right],
	\]
	and:
	\[
		\Theta_{2} = \sum_{\gamma\in\mathcal{P}_{0}^{\leq L}}\sum_{2\leq a\leq L}\mathbb{E}_{n}\left[f_{n}(\gamma,a)^{2}\right],
	\]
	so that:
	\[
		\mathbb{E}_{n}\left[\text{Diag}\right]=\frac{1}{2\pi}\left[\Theta_{1} + \Theta_{2}\right].
	\]
	\subsubsection{Estimating $\Theta_{1}$}
	Using Lemma \ref{f exp}:
	\[
		\Theta_{1} = \sum_{\gamma\in\mathcal{P}_{0}^{\leq L}}\mathbb{E}_{n}\left[f_{n}(\gamma,1)^{2}\right]=
		\sum_{\gamma\in\mathcal{P}_{0}^{\leq L}}\left[\left(\frac{\mathfrak{R}(\chi(\gamma))l_{\gamma}\hat{\psi}(l_{\gamma}/L)}{\sinh(l_{\gamma}/2)}\right)^{2} + O_{\gamma}(1/n)\right].
	\]
	$L$ is fixed when $n\to\infty$ and the constants in $O_{\gamma}(1/n)$
	depend only on $\gamma$, combining this with the fact that the sum
	is over $\gamma\in\mathcal{P}_{0}^{\leq L}$ we get that the implied constants in $O_{\gamma}(1/n)$ are bounded by a function of $L$. This gives:
	\[
		\Theta_{1}=\sum_{\gamma\in\mathcal{P}_{0}^{\leq L}}\left(\frac{\mathfrak{R}(\chi(\gamma))l_{\gamma}\hat{\psi}(l_{\gamma}/L)}{\sinh(l_{\gamma}/2)}\right)^{2} + O_{L}(1/n),
	\]
	so all that is left is to estimate the sum.

	Let $1\ll x \leq L$. As $\chi$ is a unitary character:
	\[
		\sum_{\substack{\gamma\in\mathcal{P}_{0}^{\leq L} \\ l_{\gamma}\leq x}}\left[\mathfrak{R}(\chi(\gamma)\right]^{2}=\frac{1}{4}\sum_{\substack{\gamma\in\mathcal{P}_{0}^{\leq L} \\ l_{\gamma}\leq x}}\chi(\gamma)^{2}+\overline{\chi(\gamma)}^{2}+2.
	\]

	If $\chi^{2}=1$ the prime geodesic theorem - Theorem \ref{prime geo}, gives this sum
	as $\frac{1}{2}\frac{e^{x}}{x}\left(1+O(1/x)\right)$ (we are summing
	over $\mathcal{P}_{0}$ so that we only count half of the geodesics), while if $\chi^{2}\neq1$ Proposition
	\ref{prime chi geo} gives this as $\frac{1}{4}\frac{e^{x}}{x}\left(1+O(1/x)\right)$.
	Set:
	\[
		r_{\chi}=\begin{cases}
			1/2 & \chi^{2}=1,    \\
			1/4 & \chi^{2}\neq1,
		\end{cases}
	\]
	so that:
	\[
		\sum_{\substack{\gamma\in\mathcal{P}_{0}^{\leq L} \\ l_{\gamma}\leq x}}\left[\mathfrak{R}(\chi(\gamma))\right]^{2}=r_{\chi}\frac{e^{x}}{x}\left(1+O(1/x)\right).
	\]

	Using the fact that $\sinh^{2}(x/2)=\frac{1}{4}e^{x}(1+O(1/x))$ we get:
	\[
		\sum_{\gamma\in\mathcal{P}_{0}^{\leq L}}\left(\frac{\mathfrak{R}(\chi(\gamma))l_{\gamma}\hat{\psi}(l_{\gamma}/L)}{\sinh(l_{\gamma}/2)}\right)^{2}=
		\sum_{\gamma\in\mathcal{P}_{0}^{\leq L}}\frac{\left[\mathfrak{R}(\chi(\gamma))\right]^{2}l_{\gamma}^{2}\hat{\psi}^{2}(l_{\gamma}/L)}{\sinh^{2}(l_{\gamma}/2)}=
	\]

	\[
		=4\sum_{\gamma\in\mathcal{P}_{0}^{\leq L}}\left[\mathfrak{R}(\chi(\gamma))\right]^{2}l_{\gamma}^{2}\hat{\psi}^{2}(l_{\gamma}/L)e^{-l_{\gamma}}+
		O\left(\sum_{\gamma\in\mathcal{P}_{0}^{\leq L}}l_{\gamma}\hat{\psi}^{2}(l_{\gamma}/L)e^{-l_{\gamma}}\right).
	\]
	Set:
	\[
		A = \sum_{\gamma\in\mathcal{P}_{0}^{\leq L}}\left[\mathfrak{R}(\chi(\gamma))\right]^{2}l_{\gamma}^{2}\hat{\psi}^{2}(l_{\gamma}/L)e^{-l_{\gamma}},
	\]
	and:
	\[
		B = \sum_{\gamma\in\mathcal{P}_{0}^{\leq L}}l_{\gamma}\hat{\psi}^{2}(l_{\gamma}/L)e^{-l_{\gamma}},
	\]
	so that:
	\[
		\Theta_{1} = 4A+O(B) + O_{L}(1/n).
	\]

	We first tackle $A$. Summation by parts gives:
	\begin{multline*}
		A = \sum_{\gamma\in\mathcal{P}_{0}^{\leq L}}\left[\mathfrak{R}(\chi(\gamma))\right]^{2}l_{\gamma}^{2}\hat{\psi}^{2}(l_{\gamma}/L)e^{-l_{\gamma}}=\\
		\int_{\text{Sys}(X)}^{L}r_{\chi}\frac{e^{x}}{x}(1+O(1/x))\left(x^{2}\hat{\psi}^{2}(x/L)-R'\right)e^{-x},
	\end{multline*}
	where $\text{Sys}(X)>0$ is the systole of $X$, that is, the length of the shortest closed geodesic on $X$, and
	$R(x)=x^{2}\hat{\psi}^{2}(x/L)$. In particular:
	\[
		R'=2x\hat{\psi}^{2}(x/L)+\frac{2x^{2}}{L}\hat{\psi}'(x/L).
	\]
	Write the above integral as:
	\[
		r_{\chi}\int_{\text{Sys}(X)}^{L}x\hat{\psi}^{2}(x/L)-\frac{R'}{x}+O(1/x)\left(x\hat{\psi}^{2}(x/L)-\frac{R'}{x}\right).
	\]

	Using Lagrange's Theorem as well as the fact that $\hat{\psi}$ is even we get:
	\begin{multline*}
		\int_{\text{Sys}(X)}^{L}x\hat{\psi}^{2}(x/L)=L^{2}\int_{\text{Sys}(X)/L}^{1}u\hat{\psi}^{2}(u)=\\\frac{1}{2}L^{2}\int_{\mathbb{R}}|u|\hat{\psi}^{2}(u) - \hat{\psi}^{2}(c)L^{2}\int_{0}^{\text{Sys}(X)/L}u=\frac{1}{4}L^{2}\Sigma_{\text{GOE}}^{2}(\psi) + O(1),
	\end{multline*}
	while:
	\[
		\int_{\text{Sys}(X)}^{L}\frac{R'}{x}=2\int_{\text{Sys}(X)}^{L}\hat{\psi}^{2}(x/L)+\frac{2}{L}\int_{\text{Sys}(X)}^{L}x\hat{\psi}'(x/L)\ll L.
	\]
	In addition:
	\[
		\int_{\text{Sys}(X)}^{L}O(1/x)\left(x\hat{\psi}^{2}(x/L)-\frac{R'}{x}\right)\ll L,
	\]
	so that:
	\[
		4A=4\sum_{\gamma\in\mathcal{P}_{0}^{\leq L}}\mathfrak{R}(\chi(\gamma))^{2}l_{\gamma}^{2}\hat{\psi}^{2}(l_{\gamma}/L)e^{-l_{\gamma}}=L^{2}r_{\chi}\Sigma_{\text{GOE}}^{2}(\psi)+O(L).
	\]

	A similar summation by parts argument gives:
	\[
		B\ll L,
	\]
	and so:
	\[
		4A+O(B)=\sum_{\gamma\in\mathcal{P}_{0}^{\leq L}}\left(\frac{\mathfrak{R}(\chi(\gamma))l_{\gamma}\hat{\psi}(l_{\gamma}/L)}{\sinh(l_{\gamma}/2)}\right)^{2}=L^{2}r_{\chi}\Sigma_{\text{GOE}}^{2}(\psi)+O(L).
	\]
	In particular:
	\[
		\Theta_{1} = 4A+O(B) + O_{L}(1/n) = L^{2}r_{\chi}\Sigma_{\text{GOE}}^{2}(\psi)+O(L) + O_{L}(1/n).
	\]

	\subsubsection{Bounding $\Theta_{2}$}
	Recall that:
	\[
		\Theta_{2} = \sum_{\gamma\in\mathcal{P}_{0}^{\leq L}}\sum_{2\leq a\leq L}\mathbb{E}_{n}\left[f_{n}(\gamma,a)^{2}\right].
	\]
	Using Lemma \ref{f exp} and:
	\[
		\sum_{d|a}d\ll a^{2},
	\]
	gives:

	\begin{multline*}
		\Theta_{2}=\sum_{\gamma\in\mathcal{P}_{0}^{\leq L}}\sum_{2\leq a\leq L}\mathbb{E}_{n}\left[f_{n}(\gamma,a)^{2}\right]\ll\\
		\sum_{\gamma\in\mathcal{P}_{0}^{\leq L}}\sum_{2\leq a\leq L}\left[a^{2}\left(\frac{\mathfrak{R}(\chi(\gamma^{a}))l_{\gamma}\hat{\psi}(al_{\gamma}/L)}{\sinh(al_{\gamma}/2)}\right)^{2} + O_{\gamma,a}(1/n)\right].
	\end{multline*}
	As $a,l_{\gamma}\leq L$, we have:
	\[
		\Theta_{2} \ll \sum_{\gamma\in\mathcal{P}_{0}^{\leq L}}\sum_{a\geq2}a^{2}\left(\frac{\mathfrak{R}(\chi(\gamma^{a}))l_{\gamma}\hat{\psi}(al_{\gamma}/L)}{\sinh(al_{\gamma}/2)}\right)^{2} + O_{L}(1/n).
	\]

	A simple estimate gives:
	\begin{multline*}
		\sum_{\gamma\in\mathcal{P}_{0}^{\leq L}}\sum_{a\geq2}a^{2}\left(\frac{\mathfrak{R}(\chi(\gamma^{a}))l_{\gamma}\hat{\psi}(al_{\gamma}/L)}{\sinh(al_{\gamma}/2)}\right)^{2}\ll\\
		\sum_{\gamma\in\mathcal{P}_{0}^{\leq L}}\sum_{a\geq2}a^{2}l_{\gamma}^{2}e^{-al_{\gamma}}=
		\sum_{\gamma\in\mathcal{P}_{0}^{\leq L}}l_{\gamma}^{2}\sum_{a\geq2}a^{2}e^{-al_{\gamma}}.
	\end{multline*}
	Using the fact that for $l_{\gamma}\gg1$ we have $\sum_{a\geq2}a^{2}e^{-al_{\gamma}}\ll e^{-2l_{\gamma}}$ yields:
	\[
		\sum_{\gamma\in\mathcal{P}_{0}^{\leq L}}l_{\gamma}^{2}\sum_{a\geq2}a^{2}e^{-al_{\gamma}}\ll\sum_{\gamma\in\mathcal{P}_{0}^{\leq L}}l_{\gamma}^{2}e^{-2l_{\gamma}}\ll\int_{0}^{L}xe^{-x}\ll1,
	\]
	hence:
	\[
		\Theta_{2}\ll 1+O_{L}(1/n).
	\]

	Recall that:
	\[
		\mathbb{E}_{n}\left[\text{Diag}\right]=\frac{1}{2\pi}\left[\Theta_{1}+\Theta_{2}\right].
	\]
	Plugging in our estimates for $\Theta_{1}$ and $\Theta_{2}$ we have:
	\begin{multline*}
		\mathbb{E}_{n}\left[\text{Diag}\right] = \frac{1}{2\pi}\left[L^{2}r_{\chi}\Sigma_{\text{GOE}}^{2}(\psi)+O(L) + O_{L}(1/n) + O(1)+ O_{L}(1/n)\right]=\\
		\frac{L^{2}}{2\pi}r_{\chi}\Sigma_{\text{GOE}}^{2}(\psi) + O(L) + O_{L}(1/n).
	\end{multline*}
	Note that $r_{\chi}\Sigma_{\text{GOE}}^{2}(\psi) = \frac{1}{2}\sigma_{\chi,\psi}^{2}$ so that:
	\[
		\mathbb{E}_{n}\left[\text{Diag}\right] = \frac{L^{2}}{4\pi}\sigma_{\chi,\psi}^{2} + O(L) + O_{L}(1/n),
	\]
	which is exactly Lemma \ref{exp diag}.
\end{proof}

\subsection{Estimating $\mathbb{E}_{n}\left[\text{Diag}^{2}\right]$}
\begin{proof} [Proof of Lemma \ref{exp diag2}]
	For $T\gg1$ we know:
	\[
		\text{Diag}=\frac{1}{2\pi}\sum_{\gamma\in\mathcal{P}_{0}}\sum_{a\geq1}f_{n}(\gamma,a)^{2}.
	\]
	Using the fact that $\hat{\psi}$ is compactly supported in $[-1,1]$ we get:
	\[
		\text{Diag}^{2}=\frac{1}{4\pi^{2}}\left[\sum_{\gamma,\delta\in\mathcal{P}_{0}^{\leq L}}f_{n}(\gamma,1)^{2}f_{n}(\delta,1)^{2}+
		\sum_{\gamma\in\mathcal{P}_{0}^{\leq L}}\sum_{\substack{1\leq a,b\leq L\\a+b\geq3}}f_{n}(\gamma,a)^{2}f_{n}(\delta,b)^{2}\right].
	\]
	Denote:
	\[
		\Omega_{1} = \sum_{\gamma,\delta\in\mathcal{P}_{0}^{\leq L}}f_{n}(\gamma,1)^{2}f_{n}(\delta,1)^{2},
	\]
	and:
	\[
		\Omega_{2} = \sum_{\gamma\in\mathcal{P}_{0}^{\leq L}}\sum_{\substack{1\leq a,b\leq L\\a+b\geq3}}f_{n}(\gamma,a)^{2}f_{n}(\delta,b)^{2},
	\]
	so that:
	\[
		\text{Diag}^{2} = \frac{1}{4\pi^{2}}\left[\Omega_{1} + \Omega_{2}\right].
	\]

	\subsubsection{Estimating $\mathbb{E}_{n}[\Omega_{1}]$}
	We have:
	\[
		\Omega_{1} = \sum_{\gamma,\delta\in\mathcal{P}_{0}^{\leq L}}f_{n}(\gamma,1)^{2}f_{n}(\delta,1)^{2}=\sum_{\gamma\neq\delta\in\mathcal{P}_{0}^{\leq L}}f_{n}(\gamma,1)^{2}f_{n}(\delta,1)^{2}+\sum_{\gamma\in\mathcal{P}_{0}^{\leq L}}f_{n}(\gamma,1)^{4}.
	\]
	Theorem \ref{main thm 2} implies that for $\gamma\neq \delta \in \mathcal{P}_{0}$ we have:
	\[
		\mathbb{E}_{n}\left[U_{n}(\gamma)^{2}U_{n}(\delta)^{2}\right]=\mathbb{E}_{n}\left[U_{n}(\gamma)^{2}\right]\mathbb{E}_{n}\left[U_{n}(\delta)^{2}\right] + O_{\gamma,\delta}(1/n),
	\]
	so that:
	\begin{multline*}
		\mathbb{E}_{n}\left[\sum_{\gamma\neq\delta\in\mathcal{P}_{0}^{\leq L}}f_{n}(\gamma,1)^{2}f_{n}(\delta,1)^{2}\right]=\\
		\left(\sum_{\gamma\in\mathcal{P}_{0}^{\leq L}}\mathbb{E}_{n}\left[f_{n}(\gamma,1)^{2}\right]\right)^{2}-\sum_{\gamma\in\mathcal{P}_{0}^{\leq L}}\mathbb{E}_{n}\left[f_{n}(\gamma,1)^{2}\right]^{2} + O_{L}(1/n).
	\end{multline*}
	By definition:
	\[
		\Theta_{1} = \sum_{\gamma\in\mathcal{P}_{0}^{\leq L}}\mathbb{E}_{n}\left[f_{n}(\gamma,1)^{2}\right],
	\]
	hence:
	\[
		\mathbb{E}_{n}[\Omega_{1}] =
		\Theta_{1}^{2} + \sum_{\gamma\in\mathcal{P}_{0}^{\leq L}}\mathbb{E}_{n}\left[f_{n}(\gamma,1)^{4}\right] -\sum_{\gamma\in\mathcal{P}_{0}^{\leq L}}\mathbb{E}_{n}\left[f_{n}(\gamma,1)^{2}\right]^{2} +O_{L}(1/n).
	\]

	Using Lemma \ref{f exp} we have $\mathbb{E}_{n}\left[f_{n}(\gamma,1)^{2}\right] \ll l_{\gamma}^{2}e^{-l_{\gamma}}$ so that:
	\[
		\sum_{\gamma\in\mathcal{P}_{0}^{\leq L}}\mathbb{E}_{n}\left[f_{n}(\gamma,1)^{2}\right]^{2} \ll \sum_{\gamma\in\mathcal{P}_{0}^{\leq L}}l_{\gamma}^{4}e^{-2l_{\gamma}} \ll 1.
	\]
	In addition, using Corollary \ref{main cor 2} and Corollary \ref{dist cor} in addition with \newline $f_{n}(\gamma,1) = s(\gamma,1)U_{n}(\gamma)$ we get $\mathbb{E}_{n}\left[f_{n}(\gamma,1)^{4}\right]\ll  l_{\gamma}^{4}e^{-2l_{\gamma}}$ so that:
	\[
		\sum_{\gamma\in\mathcal{P}_{0}^{\leq L}}\mathbb{E}_{n}\left[f_{n}(\gamma,1)^{4}\right] \ll \sum_{\gamma\in\mathcal{P}_{0}^{\leq L}}l_{\gamma}^{4}e^{-2l_{\gamma}} \ll 1.
	\]
	All in all:
	\[
		\mathbb{E}_{n}[\Omega_{1}] =
		\Theta_{1}^{2} + O(1) + O_{L}(1/n).
	\]

	Using our previous estimate of $\Theta_{1}$:
	\[
		\Theta_{1} = L^{2}r_{\chi}\Sigma_{\text{GOE}}^{2}(\psi)+O(L) + O_{L}(1/n),
	\]
	we get:
	\[
		\mathbb{E}_{n}[\Omega_{1}]=\Theta_{1}^{2} +O(1) + O_{L}(1/n) =
		L^{4}r_{\chi}^{2}\left(\Sigma_{\text{GOE}}^{2}(\psi)\right)^{2}+O(L^{3})+O_{L}(1/n).
	\]
	As $r_{\chi}\Sigma_{\text{GOE}}^{2}(\psi)=\frac{1}{2}\sigma_{\chi,\psi}^{2}$
	we have:
	\[
		\mathbb{E}_{n}[\Omega_{1}]=\frac{L^{4}}{4}\sigma_{\chi,\psi}^{4}+O(L^{3})+O_{L}(1/n).
	\]

	\subsubsection{Estimating $\mathbb{E}_{n}[\Omega_{2}]$}
	Recall that:
	\[
		\Omega_{2} = \sum_{\gamma,\delta\in\mathcal{P}_{0}^{\leq L}}\sum_{\substack{1\leq a,b\leq L\\a+b\geq3}}f_{n}(\gamma,a)^{2}f_{n}(\delta,b)^{2},
	\]
	Corollaries \ref{main cor 2} and \ref{dist cor} imply that for all $\gamma,\delta\in \mathcal{P}_{0}^{\leq L}$ and all positive integers $a,b$ we have\footnote{
	A proof similar to that of Lemma \ref{R bound} shows that for $\gamma\neq\delta$ we have $\mathbb{E}_{n}\left[U_{n}(\gamma^{a})^{2}U_{n}(\delta^{b})^{2}\right]\ll \sigma(a)\sigma(b) + O_{\gamma,\delta,a,b}(1/n)$, while:
	\[
		\mathbb{E}_{n}\left[U_{n}(\gamma^{a})^{2}U_{n}(\gamma^{b})^{2}\right]\ll \sigma_{3}(\gcd(a,b)) + \sigma(a)\sigma(b) + O_{\gamma,a,b}(1/n),
	\]
	where $\sigma_{3}(x) = \sum_{d|x}d^{3}$. Using $\gcd(a,b)\leq \min\{a,b\}$ and $\sigma_{3}(x) \ll d(x)x^{3}\ll x^{4}$ yields the given bound.
	}:
	\[
		\mathbb{E}_{n}\left[U_{n}(\gamma^{a})^{2}U_{n}(\delta^{b})^{2}\right]\ll a^{2}b^{2} + O_{\gamma,\delta,a,b}(1/n),
	\]
	so that:
	\[
		\mathbb{E}_{n}[\Omega_{2}] \ll \sum_{\gamma,\delta\in\mathcal{P}_{0}^{\leq L}}\sum_{\substack{1\leq a,b\leq L\\a+b\geq3}}l_{\gamma}^{2}l_{\delta}^{2}a^{2}b^{2}e^{-al_{\gamma}-bl_{\delta}} + O_{L}(1/n).
	\]

	Summation by parts gives:
	\[
		\sum_{\gamma\in\mathcal{P}_{0}^{\leq L}}\sum_{a\geq1}a^{2}l_{\gamma}^{2}e^{-al_{\gamma}}\ll L^{2},
	\]
	and:
	\[
		\sum_{\gamma\in\mathcal{P}_{0}^{\leq L}}\sum_{a\geq2}a^{2}l_{\gamma}^{2}e^{-al_{\gamma}}\ll1,
	\]
	so that:
	\[
		\sum_{\gamma,\delta\in\mathcal{P}_{0}^{\leq L}}\sum_{\substack{1\leq a,b\leq L\\a+b\geq3}}l_{\gamma}^{2}l_{\delta}^{2}a^{2}b^{2}e^{-al_{\gamma}-bl_{\delta}}\ll L^{2}.
	\]
	Thus:
	\[
		\mathbb{E}_{n}[\Omega_{2}] \ll L^{2} + O_{L}(1/n).
	\]

	Combining our estimates of $\mathbb{E}_{n}[\Omega_{1}],\mathbb{E}_{n}[\Omega_{2}]$ we have:
	\[
		\begin{split}
			\mathbb{E}_{n}[\text{Diag}^{2}] & =\frac{1}{4\pi^{2}}\left[\mathbb{E}_{n}[\Omega_{1}] + \mathbb{E}_{n}[\Omega_{2}]\right]  \\
			                                & =\frac{L^{4}}{16\pi^{2}}\sigma_{\chi,\psi}^{4}+O(L^{3})+O_{L}(1/n)+O(L^{2}) + O_{L}(1/n) \\
			                                & =\frac{L^{4}}{16\pi^{2}}\sigma_{\chi,\psi}^{4}+O(L^{3})+O_{L}(1/n).
		\end{split}
	\]
	proving Lemma \ref{exp diag2}.
\end{proof}

\bibliographystyle{abbrv}
\bibliography{EIGENVALUE_STATISTICS.bib}

\end{document}